\documentclass[a4paper]{article}
\usepackage[utf8]{inputenc}
\usepackage{amsthm}
\usepackage{authblk} 
\usepackage{hyperref}

\usepackage{todonotes}
\usepackage[english]{babel}
\usepackage{amsmath,amssymb}
\usepackage{aliascnt} 
\DeclareMathAlphabet{\mymathbb}{U}{BOONDOX-ds}{m}{n} 
\usepackage{xcolor} 
\usepackage{colortbl}
\usepackage{caption}
\usepackage{marvosym}

\DeclareMathOperator{\STAB}{STAB}
\DeclareMathOperator{\SSTAB}{SSTAB}
\DeclareMathOperator{\ThetaBody}{TH}
\DeclareMathOperator{\CompressedThetaBody}{CTH}
\DeclareMathOperator{\conv}{conv}
\DeclareMathOperator{\trace}{trace}
\DeclareMathOperator{\diag}{diag}
\DeclareMathOperator{\st}{s.t.}

\newcommand{\R}{{\mathbb R}}

\newcommand{\Nzero}{{\mathbb N}_{0}}
\newcommand{\Sym}[1]{{\mathcal S}_{#1}}
\newcommand{\simplex}{{\Delta}}
\newcommand{\allones}[1]{\mymathbb{1}_{#1}}
\newcommand{\allzeros}[1]{\mymathbb{0}_{#1}}
\newcommand{\transposedVec}{T}

\newcommand{\GI}{G_{I}}
\newcommand{\kI}{k_{I}}

\newcommand{\tI}{t_{I}}

\newcommand{\XI}{X_{I}}
\newcommand{\Jk}[1]{J_{#1}}

\newcommand{\SIi}{S^{I}_{i}}
\newcommand{\sIi}[1]{s^{I}_{#1}}
\newcommand{\SIiIndex}[1]{S^{I}_{#1}}

\newcommand{\lambdaI}{\lambda^{I}}
\newcommand{\lambdaIi}[1]{\lambdaI_{#1}}
\newcommand{\lambdaICi}[1]{{\lambdaI_{#1}}'}

\newcommand{\zESCHG}[1]{z^{\mathcal{E}}_{#1}(G)}
\newcommand{\zCESCHG}[1]{z^{\mathcal{C}}_{#1}(G)}
\newcommand{\zSESCHG}[1]{z^{\mathcal{S}}_{#1}(G)}

\newcommand{\zESCJ}{z^{\mathcal{E}}_{J}(G)}
\newcommand{\zCESCJ}{z^{\mathcal{C}}_{J}(G)}
\newcommand{\zSESCJ}{z^{\mathcal{S}}_{J}(G)}

\newcommand{\zESCJnoG}{z^{\mathcal{E}}_{J}}
\newcommand{\zCESCJnoG}{z^{\mathcal{C}}_{J}}

\newcommand{\zESCJk}{z^{\mathcal{E}}_{J_k}(G)}
\newcommand{\zCESCJk}{z^{\mathcal{C}}_{J_k}(G)}
\newcommand{\zSESCJk}{z^{\mathcal{S}}_{J_k}(G)}
\newcommand{\xOpt}{x^\ast}
\newcommand{\XOpt}{X^\ast}

\newcommand{\XOptC}{X'}
\newcommand{\XOptI}{\XOpt_I}
\newcommand{\XOptCI}{\XOptC_I}

\newcommand{\GZeroi}[1]{G^{0}_{#1}}
\newcommand{\VZeroi}[1]{V^{0}_{#1}}
\newcommand{\EZero}{E^{0}}

\newcommand{\lambdaN}{\lambda}
\newcommand{\lambdaS}{\widetilde{\lambda}}

\newtheorem{theorem}{Theorem}

\newtheorem{lemma}{Lemma}

\newtheorem{example}{Example}

\newtheorem{definition}{Definition}

\newtheorem{corollary}{Corollary}

\newtheorem{observation}{Observation}

\providecommand{\keywords}[1]
{
    \small	
    \textbf{Keywords:} #1
}

\title{On different Versions of the Exact Subgraph Hierarchy for the Stable Set 
	Problem%
	\thanks{This research was supported by the Austrian Science Fund (FWF): I 
	3199-N31 and by the Johannes
		Kepler University Linz, Linz Institute of Technology (LIT): 
		LIT-2021-10-YOU-216.}
}

\author[1]{Elisabeth Gaar$^{(\text{\Letter})}$}
\affil[1]{University of Augsburg, Germany 
	\& Johannes Kepler University Linz, Austria, 
    \href{mailto:elisabeth.gaar@uni-a.de}{elisabeth.gaar@uni-a.de}}
\date{}

\begin{document}
    
\def\myAbstract{
	Let $G$ be a graph with $n$ vertices and $m$ edges.
	One of several hierarchies towards the stability number of $G$ is the exact 
	subgraph hierarchy (ESH).
	On the first level it computes 
	the Lov\'{a}sz theta 
	function $\vartheta(G)$
	as semidefinite program (SDP) with a matrix variable of order $n+1$ and 
	$n+m+1$ constraints. 
	On the $k$-th level it adds all exact subgraph constraints (ESC) 
	for subgraphs of order $k$ to the SDP.
	An ESC ensures that the submatrix of the 
	matrix variable corresponding to the subgraph is in the correct 
	polytope.
	By including only some ESCs into the SDP
	the ESH  
	can be exploited  computationally.
	
	In this paper we introduce a variant of the ESH that 
	computes $\vartheta(G)$
	through an SDP with a 
	matrix variable of order $n$ and $m+1$ constraints.
	We show that it makes sense to include the ESCs into this SDP and 
	introduce the compressed ESH (CESH) analogously to the ESH.
	Computationally the CESH seems favorable as the SDP is smaller. 
	However, we prove that the bounds based on the ESH are always at 
	least as good as those of the CESH.  
	In computational experiments sometimes they are significantly better.
	
	We also introduce scaled ESCs (SESCs), 
	which are a more natural way to include exactness 
	constraints into the 
	smaller SDP and we prove that including an SESC is equivalent to 
	including an ESC for every subgraph.
}

\maketitle
\begin{abstract}%
		\myAbstract
\end{abstract}


\section{Introduction}
\label{sec:intro}

One of the most fundamental problems in combinatorial optimization is the 
stable set problem.
Given a graph $G = (V,E)$, a subset of vertices $S \subseteq V$ is called  
stable set if no two vertices of $S$ are adjacent. 
A stable set is called maximum stable set if there is no stable set with larger 
cardinality. 
The cardinality of a maximum stable set is called stability number of $G$ and 
denoted by $\alpha(G)$.
The stable set problem asks for a stable set of size $\alpha(G)$.  
It is an NP-hard and well-studied problem, see for example 
the survey of Bomze, Budinich, 
Pardalos and Pelillo~\cite{SurveyCliqueProblem}.

Typically NP-complete combinatorial optimization problems are solved using 
branch-and-bound or branch-and-cut algorithms.   
One type of relaxations used in order to obtain bounds are those based on 
semidefinite programming (SDP),
see Helmberg~\cite{HelmbergSdpBible} 
for an introduction.
SDPs can be solved to arbitrary precision in polynomial time 
and numerous SDP solvers are available.

Lov\'{a}sz~\cite{LovaszStart} laid the foundations 
for SDP relaxations in 1979 by introducing the Lov\'{a}sz theta function 
$\vartheta(G)$ of a graph $G$, which fulfills
\begin{equation*}
\alpha(G) \leqslant \vartheta(G) \leqslant \chi(\overline{G})
\end{equation*}
for every graph 
$G$, where $\chi(\overline{G})$ is the chromatic number 
of the complement graph~$\overline{G}$ of $G$. 
Among his formulations of $\vartheta(G)$ 
was an SDP with a matrix variable of order $n$ and $m+1$ equality 
constraints. We will give the rigorous definition of this SDP 
in Section~\ref{sec:Compare2ThetaFunction}
and refer to it as~\eqref{lovaszTheta_n}.
As a result, $\vartheta(G)$ can be calculated in 
polynomial time, even though 
it is sandwiched between $\alpha(G)$ and $\chi(\overline{G})$, which are both 
NP-complete to compute.
Later Gr{\"o}tschel, Lov{\'a}sz and 
Schrijver~\cite{OurUsedFormOfLovasTheta} provided an alternative formulation 
of  $\vartheta(G)$ as 
SDP with a matrix variable of order $n+1$ and $n+m+1$ equality constraints.
In Section~\ref{sec:DefStab} we will give the rigorous definition of this 
SDP and refer to it as~\eqref{lovaszTheta_nPlus1}.

In 1995 Goemans and Williamson~\cite{GoemansWilliamson} presented an SDP 
relaxation for the Max-Cut problem which is a provenly good 
approximation.
Since then SDP relaxations have been used for various 
combinatorial optimization problems and several ways of further 
tightening them have been developed.
Also  
hierarchies, that consist of several levels, were established, for example 
by Lov\'asz and 
Schrijver~\cite{LovaszSchrijverHierarchy} and by 
Lasserre~\cite{LasserreHierarchy}. 
At the first level a simple relaxation is considered, 
and the higher the level gets, 
the tighter the bounds become. 
Usually the computational power it takes to evaluate the level of the 
hierarchy increases on each level and 
often the computation of higher levels is beyond reach.
They major drawback of most of the SDP based hierarchies is 
that the order of the matrix variable increases 
enormously with each level.

In 2015 Adams, Anjos, Rendl and Wiegele~\cite{AARW} introduced the exact 
subgraph hierarchy (ESH) for combinatorial optimization problems that have an 
SDP relaxation. 
They discussed the ESH for the Max-Cut problem 
and briefly described it for the stable set problem.
Here the first level of the hierarchy boils down 
to~\eqref{lovaszTheta_nPlus1}.
They introduced exact subgraph constraints (ESC), 
which ensure
that the submatrix of the matrix variable in~\eqref{lovaszTheta_nPlus1} 
corresponding to a subgraph 
is in the so-called squared stable set polytope. 
If the problem is solved exactly the submatrix has to be in this polytope, 
hence the ESC forces the subgraph to be exact. 
On the $k$-th level of the ESH the ESC for all subgraphs of order $k$ 
are included into~\eqref{lovaszTheta_nPlus1}.
This implies that the order of the matrix variable  
remains 
$n+1$ 
on each level of the ESH. 
Gaar and 
Rendl~\cite{elli-diss, GaarRendl, GaarRendlFull} 
computationally exploit the ESH and 
relaxations of it for the 
stable set, the Max-Cut  
and the coloring problem.

To summarize, the ESH from~\cite{AARW} starts from~$\vartheta(G)$ 
formulated as~\eqref{lovaszTheta_nPlus1} and adds ESCs on higher 
levels.
As $\vartheta(G)$ has two SDP 
formulations~\eqref{lovaszTheta_nPlus1} and~\eqref{lovaszTheta_n}, it is 
a natural question whether it makes sense to build a hierarchy by starting 
from $\vartheta(G)$ formulated as~\eqref{lovaszTheta_n} 
and adding ESCs.
It is the aim of this paper to investigate this natural question, which is 
even more interesting in the light of a 
a recent work by
Galli and Letchford~\cite{GalliLetchford}, who compared the behavior 
of~\eqref{lovaszTheta_nPlus1} and~\eqref{lovaszTheta_n} when they are 
strengthened or weakened and who showed that the obtained bounds do not 
always 
coincide.

In this paper we show that 
it makes sense to consider this new hierarchy, which we newly introduce as
 compressed (because the SDP is smaller) ESH (CESH).
We prove that both the ESH and the 
CESH are equal to $\vartheta(G)$ on the first level and equal to $\alpha(G)$ on 
the $n$-th level.
Furthermore, the SDP has a smaller matrix variable and and fewer 
constraints, so intuitively the CESH is computationally favorable.
However, we prove that the bounds obtained by including an ESC 
into~\eqref{lovaszTheta_nPlus1} are always at least as good as those obtained 
from including the same ESC into~\eqref{lovaszTheta_n}, demonstrating that 
the bounds obtained from the ESH are at least as good as those from the CESH.
Furthermore, it turns out in our computational comparison that the 
bounds 
are 
sometimes significantly worse for the CESH, 
but the running times do not significantly 
decrease. Hence, we confirm that the ESH has the better trade-off 
between the quality of the bound and the running time.

The intuition behind the SDP~\eqref{lovaszTheta_n} is different than the one 
of~\eqref{lovaszTheta_nPlus1}, in particular for the solutions 
 representing stable sets. 
We show in this paper that there is an alternative intuitive definition of 
exact subgraphs 
for~\eqref{lovaszTheta_n}. 
This leads to our new definition of scaled ESCs (SESC)
and our introduction of another new hierarchy, the scaled ESH (SESH).
We prove that SESCs coincide with the 
original ESCs for~\eqref{lovaszTheta_n},
which implies that the ESH and the SESH coincide.

To summarize, in this paper we confirm that even though our new 
hierarchies based on exactness
seem more intuitive and computational favorable, with off the 
shelve SDP solvers it is the 
best option to consider the ESH in the way it has been done so far.
Our findings are in accordance with the results of~\cite{GalliLetchford}, 
where it is observed that~\eqref{lovaszTheta_nPlus1} typically gives stronger 
bounds when strengthened.

The rest of the paper is organized as follows.
In Section~\ref{sec:eshOriginal} we give rigorous definitions of ESCs and the 
ESH and explain how they can be exploited computationally.
In Section~\ref{esc:compressedESH} we introduce the CESH 
and compare it to the ESH, also in the light of the results 
of~\cite{GalliLetchford}.
Then we introduce SESCs in Section~\ref{sec:scaledESC} and 
investigate how they are related to the ESCs.
In Section~\ref{sec:computations} we present computational results and 
 we conclude our paper in Section~\ref{sec:conclusions}.

We use the following notation. 
We denote by $\Nzero$ the natural numbers starting with $0$.
By $\allones{d}$ and $\allzeros{d}$ we denote the vector or matrix of all ones 
and all zeros of size $d$, respectively. Furthermore, by $\Sym{n}$ we 
denote 
the set of symmetric matrices in $\mathbb{R}^{n \times n}$.
We denote the convex hull of a set $S$ by $\conv(S)$
and the trace of a matrix $X$ by $\trace(X)$. 
Moreover, $\diag(X)$ extracts the main diagonal of the matrix $X$ into a 
vector.
By $x^{\transposedVec}$ and $X^{\transposedVec}$ we denote the transposed of 
the vector $x$ and the matrix $X$, respectively.
Moreover, we denote the $i$-th entry of the vector $x$ by $x_{i}$
and the  entry of $X$ in the 
$i$-th row and the $j$-th column by $X_{i,j}$.
Furthermore, we denote the inner product of two vectors $x$ and $y$  by 
$\left\langle x, y \right \rangle = x^{\transposedVec}y$. The inner 
product of two matrices $X = (X_{i,j})_{\substack{1 \leqslant i,j \leqslant 
n}}$ and $Y = (Y_{i,j})_{\substack{1 \leqslant i,j \leqslant n}}$ is defined as
$
\left\langle X,Y \right\rangle = \sum_{i = 1}^{n} \sum_{j = 1}^{n} 
X_{i,j}Y_{i,j}
$.
Furthermore, the $t$-dimensional simplex is given as
$
\simplex_{t} = \left\{\lambda \in \R^{t}:
\sum_{i=1}^{t} \lambda_{i} = 1,~
 \lambda_i \geqslant 0 \quad \forall 1 \leqslant i \leqslant t\right\}
$.

\section{The Exact Subgraph Hierarchy}
\label{sec:eshOriginal}

In this section we recall exact subgraph constraints and the exact 
subgraph hierarchy for combinatorial optimization problems that have 
an SDP relaxation introduced  
by Adams, Anjos, Rendl and 
Wiegele in 2015~\cite{AARW}. 
We detail everything for the stable set problem, because 
in~\cite{AARW} they 
focused on Max-Cut. %
Besides motivation and definitions, we provide new examples, discuss the 
representation of exact subgraph 
constraints and compare the exact 
	subgraph hierarchy to other hierarchies from the literature.

\subsection{Lov\'asz Theta Function}
\label{sec:DefStab}

We start by presenting the Lov\'asz theta function. To to so,
it is handy to consider 
the 
incidence vectors of stable sets and the polytope they span.

\begin{definition}
    \label{def:STAB}
    Let $G = (V,E)$ be a graph with  $|V|=n$ and $V = \{1, \dots, n\}$. Then 
    the set of all 
    stable set 
    vectors  $\mathcal{S}(G)$ and the stable set polytope $\STAB(G)$ are 
    defined as 
    \begin{align*}
    \mathcal{S}(G) &= \left\{ s \in \{0,1\}^{n} : s_{i}s_{j} = 0 \quad \forall 
    \{i,j\} \in E\right\} \text{ and}\\ 
    \STAB(G) &= \conv\left\{s : s \in \mathcal{S}(G)\right\}.
    \end{align*}  
\end{definition}

It is easy to see that the stability number $\alpha(G)$ is obtained by 
solving 
\begin{equation*}
\alpha(G) = \max_{s \in \mathcal{S}(G)} \allones{n}^{\transposedVec}s = \max_{s 
\in \STAB(G)} \allones{n}^{\transposedVec}s,
\end{equation*}
but unfortunately $\STAB(G)$ is very hard to describe in general.
Several linear relaxations of $\STAB(G)$ have been considered, 
like the so-called 
fractional stable set polytope and the clique constraint stable set polytope. 
We refer to~\cite{OurUsedFormOfLovasTheta} for further details.

We focus on another relaxation, namely the Lov\'asz theta function 
$\vartheta(G)$, which is an upper bound on~$\alpha(G)$.
Gr\"otschel, Lov\'asz and 
Schrijver~\cite{OurUsedFormOfLovasTheta} proved 
\begin{alignat}{3}
\tag{$T_{n+1}$}
\label{lovaszTheta_nPlus1}
\vartheta(G) =
& \max & \allones{n}^{\transposedVec}x \phantom{iiii}  \\ \nonumber
& \st &  \diag(X) &= x\\ \nonumber
&& X_{i,j} &= 0 && \forall \{i,j\} \in E \\ \nonumber                  
&& \left(\begin{array}{cc}
1  & x^{\transposedVec} \\
x & X
\end{array}\right) &\succcurlyeq 0\\ \nonumber
&& X \in \Sym{n}&,~x \in \mathbb{R}^{n}
\end{alignat}
and hence provided an SDP formulation of $\vartheta(G)$.
This SDP has a matrix variable of order $n+1$. 
Furthermore, there are $m$ constraints of the form $X_{i,j} = 0$, $n$ 
constraints to make sure that $\diag(X) = x$ 
 and one constraint ensures 
that in the matrix of order $n+1$ the entry in 
 the first row and 
first column is equal to~$1$. Hence, there are $n+m+1$ linear equality 
constraints in~\eqref{lovaszTheta_nPlus1}. 

To formulate~\eqref{lovaszTheta_nPlus1} in a more compact 
way we observe the well-known fact that 
$ X-xx^{\transposedVec} \succcurlyeq 0$ if and only if 
$\left(\begin{array}{cc}
1  & x^{\transposedVec} \\
x & X
\end{array}\right) \succcurlyeq 0$, see 
Boyd and Vandenberghe~\cite[Appendix 
A.5.5]{BoydVandenbergheConvexOptimization}
on Schur complements. 
Thus, the feasible region of~\eqref{lovaszTheta_nPlus1} is
\begin{align*}
\ThetaBody^{2}(G) &= \left\{ (x,X) \in \R^{n}\times\Sym{n} \colon \diag(X) = x, 
\right.\\ 
& \left. \qquad \qquad X_{i,j}=0 \quad \forall \{i,j\} \in E, \quad
X-xx^{\transposedVec} 
\succcurlyeq 0 
\right\}.
\end{align*}
Clearly for each element $(x,X)$ of $\ThetaBody^{2}(G)$ the 
projection of $X$ onto its main diagonal is $x$. The set of all projections 
\begin{equation*} 
\ThetaBody(G) = \left\{ x \in \R^{n}\colon \exists X \in \Sym{n} : (x,X) \in 
\ThetaBody^{2}(G) \right\}
\end{equation*}
is called theta body. More information on $\ThetaBody(G)$ can be found for 
example in Conforti, Cornuejols and 
Zambelli~\cite{ConfortiCornuejolsZambelliIntegerProgramming}.
It is easy to see that
$
\STAB(G) \subseteq \ThetaBody(G)
$
holds for every graph $G$, see~\cite{OurUsedFormOfLovasTheta}. 
Thus, $\vartheta(G)$ is a relaxation of $\alpha(G)$.

\subsection{Introduction of the Exact Subgraph Hierarchy}
\label{sec:DefESH}

In order to present the exact subgraph hierarchy we need a modification of 
the stable set polytope $\STAB(G)$, namely the squared stable set polytope.

\begin{definition}
    \label{def:STAB2}
    Let $G = (V,E)$ be a graph. The squared stable set polytope 
    $\STAB^{2}(G)$ of $G$ is defined as
    \begin{equation*}
    \STAB^{2}(G) = \conv\left\{ss^{\transposedVec}: s \in 
    \mathcal{S}(G)\right\}.
    \end{equation*}
    The matrices of the form $ss^{\transposedVec}$ for $s \in 
    \mathcal{S}(G)$ are called stable set matrices.
\end{definition}

Note that the elements of $\STAB(G)$ are vectors in 
$\R^{n}$, whereas the elements of $\STAB^{2}(G)$ are matrices in  $\R^{n\times 
n}$.
In comparison to $\STAB(G)$ the structure of $\STAB^{2}(G)$ is more 
sophisticated and less studied. 
Only if $G$ has no edges  
a projection 
of $\STAB^{2}(G)$ coincides with 
a well-studied object, the boolean quadric 
polytope, see  
Padberg~\cite{PadbergBQP}.
In particular,  
by putting the upper triangle with the main diagonal into a vector for 
all 
elements of $\STAB^{2}(G)$ we obtain the elements of the boolean quadric 
polytope.

Let us now turn back to $\vartheta(G)$. The 
following lemma  
turns out to be the key ingredient for defining the exact subgraph hierarchy.

\begin{lemma}
    \label{lem:SdpWithESCOfSizeNIsAlpha}
    If we add the constraint $X \in \STAB^{2}(G)$ 
    into~\eqref{lovaszTheta_nPlus1} for a graph~$G$, 
    then the optimal objective function value is $\alpha(G)$, so
    \begin{equation}
    \label{eq:SdpWithESCOfSizeNIsAlpha}
    \alpha(G) = 
     \max \left\{ 
    \allones{n}^{\transposedVec}x :
    (x,X) \in \ThetaBody^2(G),~  
    X \in \STAB^{2}(G) \right\}. 
    \end{equation}
\end{lemma}
\begin{proof}
	Let $(P^{\mathcal{E}})$ be  the SDP on the 
	right-hand side of~\eqref{eq:SdpWithESCOfSizeNIsAlpha}, 
	let $z^{\mathcal{E}}$ be its optimal objective function value 
	and let $\mathcal{S}(G) = \{s_{1}, \dots, s_{t}\}$. 
	
	Let without loss of generality $s_t$ be the incidence vector of a maximum 
	stable set of $G$. Then clearly $x = s_t$ and $X = s_ts_t^{\transposedVec}$ 
	is feasible for $(P^{\mathcal{E}})$ and has objective function value 
	$\alpha(G)$, 
	so $\alpha(G) \leqslant z^{\mathcal{E}}$ holds.
	
	Furthermore, any feasible solution $(x,X)$ of $(P^{\mathcal{E}})$ can be 
	written as 
	\begin{equation*}
	X = \sum_{i=1}^{t} \lambda_{i}s_{i}s_{i}^{\transposedVec}
\end{equation*} 
for some $\lambda \in \simplex_{t}$ because $X \in \STAB^{2}(G)$ holds.
Thus, $x$ can be written as 
\begin{equation*}
x = \diag(X) = \diag\left(\sum_{i=1}^{t} 
\lambda_{i}s_{i}s_{i}^{\transposedVec}\right) = \sum_{i=1}^{t} \lambda_{i}s_{i}.
\end{equation*}
In consequence, the objective function value of $(x,X)$ for $(P^{\mathcal{E}})$ 
is 
equal to
\begin{equation*}
\allones{n}^{\transposedVec}x
= 
\allones{n}^{\transposedVec}\sum_{i=1}^{t} \lambda_{i}s_{i}
=
\sum_{i=1}^{t} \lambda_{i}\allones{n}^{\transposedVec}s_{i}
\leqslant 
\sum_{i=1}^{t} \lambda_{i}\alpha(G)
=
\alpha(G)
\end{equation*}
and hence $z^{\mathcal{E}} \leqslant \alpha(G)$ holds, which finishes the proof.
\end{proof}

Lemma~\ref{lem:SdpWithESCOfSizeNIsAlpha} implies that if we add the constraint 
$X \in \STAB^{2}(G)$ to~\eqref{lovaszTheta_nPlus1}, then we get the best 
possible bound on 
$\alpha(G)$,
namely $\alpha(G)$. 
Unfortunately, depending on the representation of the constraint, we either 
include an 
exponential number of new variables 
(if we use a formulation as convex hull) or 
inequality constraints 
(if we include inequalities representing facets of $\STAB^2(G)$, 
see Section~\ref{sec:ESCAsInequalities}) 
into the SDP.
In order to only partially include  
$X \in \STAB^{2}(G)$ we exploit a property of 
stable sets, namely that  
a stable set of~$G$ induces also a stable set in each subgraph 
of $G$. 
To formalize this in an observation, we first need the following 
definition.

\begin{definition}
    Let $I \subseteq V$ be a subset of the vertices of the graph $G = (V,E)$ 
    with 
    $|V| = n$ and let $\kI = |I|$. 
    We denote by $\GI$ the subgraph of $G$ that is induced by~$I$.  
    Furthermore, we denote by $\XI=(X_{i,j})_{i,j\in I}$ the submatrix of 
    $X\in\R^{n\times n}$ which is indexed by~$I$.
\end{definition}

\begin{observation}
    \label{lem:XInSTABGImpliesXIInSTABGI}
    Let $G = (V,E)$ be a graph. Then
    \begin{equation*}
    X \in \STAB^{2}(G) \quad \Leftrightarrow \quad \XI \in \STAB^{2}(\GI) \quad 
    \forall I \subseteq V.
    \end{equation*}
\end{observation}
\begin{proof}
    As $\XI \in \STAB^{2}(\GI)$ for all $I \subseteq V$ implies 
    $X \in \STAB^{2}(G)$ for $I=V$, one direction of the 
    equivalence is trivial.
    For the other direction note that $X \in \STAB^{2}(G)$ implies 
    that $X$ is a convex combination of $ss^T$ for stable set vectors $s \in 
    \mathcal{S}(G)$. 
    From this one can easily extract a convex combination of 
    $ss^T$ for $s \in \mathcal{S}(\GI)$ for $\XI$, thus 
    $\XI \in \STAB^{2}(\GI)$ for all $I \subseteq V$.
\end{proof}

Observation~\ref{lem:XInSTABGImpliesXIInSTABGI} implies that adding the 
constraint $X \in \STAB^{2}(G)$ to~\eqref{lovaszTheta_nPlus1} as 
in 
Lemma~\ref{lem:SdpWithESCOfSizeNIsAlpha} makes sure that the constraint $\XI 
\in 
\STAB^{2}(\GI)$ is fulfilled for all subgraphs $\GI$ of $G$. 
This gives rise to the following definition.

\begin{definition}
    \label{def:ESC}
    Let $G = (V,E)$ be a graph and let $I \subseteq V$.
    Then the exact subgraph constraint (ESC) for $\GI$ is defined as
    $
    \XI \in \STAB^{2}(\GI).
    $
\end{definition}

Finally we consider the hierarchy by 
Adams, Anjos, Rendl and Wiegele~\cite{AARW}.
\begin{definition}
    Let $G = (V,E)$ be a graph with $|V| = n$ and
    let $J$ be a set of subsets of $V$. Then $\zESCJ$ is the optimal objective 
    function value
    of~\eqref{lovaszTheta_nPlus1} with the ESC for 
    every 
    subgraph induced by a set in $J$, so
    \begin{equation}
    \label{SDPRelaxationWithESCHierarchyJ}
    \zESCJ = 
    \max \left\{
    \allones{n}^{\transposedVec}x:
    (x,X) \in \ThetaBody^2(G),~
    \XI \in \STAB^{2}(\GI) \quad \forall I \in J
    \right\}. 
    \end{equation}
    
    Furthermore, for $k \in \Nzero$ with $k \leqslant n$
    let $\Jk{k} = \{I \subseteq V: |I| = k\}$. 
    Then the $k$-th level of the exact subgraph hierarchy (ESH) is defined as
    $\zESCHG{k} = \zESCJk$.
\end{definition}

In other words the $k$-th level of the ESH is the 
SDP for calculating the Lov\'{a}sz theta function~\eqref{lovaszTheta_nPlus1} 
with additional ESCs for every 
subgraph of order $k$. Due to Lemma~\ref{lem:XInSTABGImpliesXIInSTABGI} every 
level of the ESH is a relaxation of \eqref{eq:SdpWithESCOfSizeNIsAlpha}.

Note that Adams, Anjos, Rendl and Wiegele did not give the hierarchy a name. 
However, they  called 
the ESCs for all subgraphs of order $k$ and therefore the 
constraint to add at the $k$-th level of the ESH the $k$-projection 
constraint.

Let us briefly look at some properties of $\zESCHG{k}$.
For example, the next lemma shows that 
	the bound obtained from the ESH is better the higher the level of the 
ESH is.

\begin{lemma}
    \label{lem:PropertiesESH}
    Let $G = (V,E)$ be a graph with $|V| = n$. 
    Then
\begin{equation*}
\vartheta(G) = \zESCHG{0} =  \zESCHG{1} \geqslant \zESCHG{k-1} \geqslant 
\zESCHG{k} \geqslant \zESCHG{n} = \alpha(G)
\end{equation*}
    holds for all $k \in \{1, \dots, n\}$.
\end{lemma}
\begin{proof}
 Lemma~\ref{lem:SdpWithESCOfSizeNIsAlpha} states that $\zESCHG{n} = 
 \alpha(G)$.
For $k = 0$ we do not add any additional constraint 
into~\eqref{lovaszTheta_nPlus1}. 
For $k=1$ the ESC for $I = 
\{i\}$ boils down to $X_{i,i} \in [0,1]$, which is enforced by 
$X \succcurlyeq 0$.
Therefore, $\vartheta(G) = \zESCHG{0} = \zESCHG{1}$ holds.
Additionally, due to Lemma~\ref{lem:XInSTABGImpliesXIInSTABGI} whenever 
all 
subgraphs of order $k$ 
are exact, also all subgraphs of order $k-1$ are exact, 
which yields the desired result. 
\end{proof}

Next, we consider an example in order to get a feeling for the ESH and 
how 
good the bounds on $\alpha(G)$ obtained with it are.

\begin{example}
    \label{Ex:ESHforThreeGraphs}
    We consider $\zESCHG{k}$ for $k \leqslant 8$ for a Paley graph, a Hamming 
    graph~\cite{DIMACS1992} and a 
    random 
    graph $G_{60,0.25}$ from the Erd\H{o}s-R\'enyi 
    model 
    in Table~\ref{tab:CalculateESHForThreeGraphs}.
    It is possible to compute 
    $\zESCHG{2}$. For $k \geqslant 3$ we use  
    relaxations (i.e.\ we compute $\zESCJ$ by including the ESCs only for a 
    subset $J$ of the 
    set of all subgraphs of order $k$ 
    and determine the sets $J$ as it is described in more detail in 
    Section~\ref{sec:computations}) 
    to get an upper bound on $\zESCHG{k}$ or deduce the value.

    For hamming6\_4 
    already for $k=2$ the upper 
    bound $\zESCHG{k}$ matches $\alpha(G)$. Thus,  
    $\zESCHG{k} = 4$ holds for all $k \geqslant 2$, so $\zESCHG{k}$ 
    is an excellent bound on $\alpha(G)$ for this graph.
    For $G_{60,0.25}$ as $k$ increases $\zESCHG{k}$ improves 
    little by little. 
    For $k = 4$ the floor value of $\zESCHG{k}$ decreases, which is 
    very important in a branch-and-bound framework, where this 
    potentially reduces the size of the branch-and-bound tree 
    drastically.
    For the Paley graph on 61 vertices only for $k \geqslant 6$ the value of 
    $\zESCHG{k}$ improves towards $\alpha(G)$. This example 
    represents one of the worst cases, where including ESCs for subgraphs of 
    small order does not give an improvement of the upper bound.

    \begin{table}[hbtp]
        \caption{The value of $\zESCHG{k}$ for three graphs. Values 
            in gray cells are only upper bounds on $\zESCHG{k}$}
        \label{tab:CalculateESHForThreeGraphs}
        \begin{center}
            \begin{footnotesize}
                \setlength\tabcolsep{3pt} 
                \renewcommand{\arraystretch}{1.25}
                \begin{tabular}{|c|r|r|r|r|r|r|r|r|r|r|}
                    \hline
                    $G$ & $n$ & $\alpha(G)$ & $\vartheta(G)$ & $\zESCHG{2}$ & 
                    $\zESCHG{3}$ &$\zESCHG{4}$ & $\zESCHG{5}$ &$\zESCHG{6}$ & 
                    $\zESCHG{7}$ & $\zESCHG{8}$\\ \hline 
                    hamming6\_4  &  64 &  4 & 5.333 & 4.000 & 
                    4.000 & 4.000 & 
                    4.000 & 4.000 & 
                    4.000 & 4.000 \\   
                    $G_{60,0.25}$ &  60 & 13 
                    & 14.282 
                    & 14.201 
                    & 14.156\cellcolor{gray!25}
                    & 13.945\cellcolor{gray!25}
                    & 13.741\cellcolor{gray!25}
                    & 13.386\cellcolor{gray!25}
                    & 13.209\cellcolor{gray!25} 
                    & 13.112\cellcolor{gray!25}\\ 
                    Paley61 &  61 &  5 
                    & 7.810 
                    & 7.810
                    & 7.810
                    & 7.810
                    & 7.810
                    & 7.078\cellcolor{gray!25}
                    & 6.989\cellcolor{gray!25}
                    & 6.990\cellcolor{gray!25}\\ \hline 
                \end{tabular}
            \end{footnotesize}
        \end{center}
        \renewcommand{\arraystretch}{1}
    \end{table}
\end{example}

Example~\ref{Ex:ESHforThreeGraphs} shows that there are 
graphs where including ESCs for subgraphs of small order improves the bound 
very much, little by little and not at all. 
It is not surprising that the ESH does not give outstanding bounds 
for all instances, as the stable set problem is NP-hard.

\subsection{Representation of Exact Subgraph Constraints}
\label{sec:ESCAsInequalities}

Next, we briefly discuss the implementation of ESCs.
In Definition~\ref{def:STAB2}
we introduced $\STAB^{2}(G)$ as convex hull, 
so the most natural way to formulate the ESC is as a convex 
combination as in the proof 
of Lemma~\ref{lem:SdpWithESCOfSizeNIsAlpha}.
We start with the following definition.
\begin{definition}
    Let $G$ be a graph and let $\GI$ be the subgraph induced by $I 
    \subseteq V$.
    Furthermore, let $|\mathcal{S}(\GI)| = \tI$ and let $\mathcal{S}(\GI) 
    = 
    \left\{\sIi{1}, \dots, \sIi{\tI}\right\}$.
    Then the $i$-th stable set matrix $\SIi$ of $\GI$ is defined 
    as $\SIi = \sIi{i}(\sIi{i})^{\transposedVec}$.
\end{definition}

Now the ESC $\XI \in \STAB^{2}(\GI)$ can be rewritten as
\begin{equation*}
\XI \in  \conv \left\{ \SIi: 
1 \leqslant i \leqslant \tI \right\}
\end{equation*}
and it is natural to implement the ESC for subgraph $\GI$ as
\begin{equation*}
\XI = \sum_{i=1}^{\tI} 
\lambdaIi{i}\SIi, \quad \lambdaI \in \simplex_{\tI}.
\end{equation*}

This implies that for the implementation of the ESC for $\GI$ we 
include $\tI$ additional 
non-negative variables, one additional equality constraint for 
$\lambdaI$ and a matrix equality constraint of size $\kI \times \kI$ that 
couples $\XI$ 
and 
$\lambdaI$ into~\eqref{lovaszTheta_nPlus1}.

There is also a different possibility to represent ESCs that uses the following 
fact.
The polytope $\STAB^{2}(\GI)$ is given by its extreme points, which 
are the stable set matrices of $\GI$. Due to the Minkowski-Weyl's theorem it 
can also be represented by its facets, i.e.\ by (finitely many) inequalities.
A priory different subgraphs induce different stable set matrices and hence 
also different squared stable set polytopes. The next result allows us to 
consider the squared stable set polytope of only one graph for a given order.
\begin{lemma}
    \label{lem:XIInSTAB2GIIffXIInSTAB2EmptyGraph}
    Let $G = (V,E)$ be a graph with $|V| = n$.
    Let $\GZeroi{n} = (\VZeroi{n},\EZero)$ with $\VZeroi{n} = \{1, 
    \dots, n\}$ and $\EZero = \emptyset$.
    Let $X \in \Sym{n}$.
    If $X_{i,j} = 0 $ for all $\{i,j\} \in E$, then
    \begin{equation*}
    X \in \STAB^{2}(G) 
    \quad \Leftrightarrow \quad 
    X \in \STAB^{2}(\GZeroi{n}).
    \end{equation*}
\end{lemma}
\begin{proof}
    If $X \in \STAB^{2}(G)$, 
    then by definition $X$ 
    is a convex combination of stable set matrices of $G$. Then it 
    is also a convex combination of stable set matrices of $\GZeroi{n}$, 
    which are all possible stable set matrices of order $n$. Hence, $X 
    \in \STAB^{2}(\GZeroi{n})$.
    
    If $X \in \STAB^{2}(\GZeroi{n})$, then $X$ is a convex 
    combination of 
    all possible stable set matrices of order $n$.
    Consider an edge $\{i,j\} \in E$, then by assumption 
    $X_{i,j} = 0$.
    Since all entries of stable set matrices are $0$ or $1$, this implies 
    that whenever the entry $(i,j)$ of a stable set matrix in the 
    convex combination is not equal to zero, 
    its coefficient is zero.
    Therefore, in the convex combination only stable set matrices which 
    are 
    also stable set matrices of $G$ have non-zero coefficients and thus 
    $X \in \STAB^{2}(G)$.
\end{proof}

As a consequence of Lemma~\ref{lem:XIInSTAB2GIIffXIInSTAB2EmptyGraph}
we can replace the ESC
$\XI \in \STAB^{2}(\GI)$ by the constraint 
$\XI \in \STAB^{2}(\GZeroi{\kI})$ whenever we add the ESC 
to~\eqref{lovaszTheta_nPlus1}. 
Thus, it is enough to have a facet representation of $\STAB^{2}(\GZeroi{\kI})$  
in order to include the ESC for $\GI$ represented by inequalities 
into~\eqref{lovaszTheta_nPlus1}.

In order to obtain all facets of $\STAB^{2}(\GZeroi{k})$ for a given $k$ we can 
use the fact that a projection of $\STAB^{2}(\GZeroi{k})$ is the 
boolean quadric polytope of size $k$ as already explained in 
Section~\ref{sec:DefESH}.  
Deza and Laurent~\cite{DezaLaurentBQP} called the boolean quadric polytope of 
size $k$ the correlation polytope of size $k$. They showed that the correlation 
polytope of size $k$ is in one-to-one correspondence with the cut polytope of 
size $k+1$ via the so-called covariance map.
Moreover, they presented a complete list of the facets of the cut 
polytopes up 
to a size of $k+1 = 7$, gave several references of other lists of facets and 
furthermore linked to a web page.
The recent version of this web page is maintained by 
Christof~\cite{HPCutFacets} and a conjectured complete facet description of the 
cut polytope of size $k+1 = 8$ and a possibly complete description of the cut 
polytope of size $k+1 = 9$ can be found there.
Therefore, we could take this list and go back via the covariance map to 
transfer it into a complete list of facets of $\STAB^{2}(\GZeroi{k})$.

However, we take a more direct path and use the software 
PORTA~\cite{PORTA} in order to obtain all inequalities that 
represent facets of $\STAB^{2}(\GZeroi{k})$ from its extreme points  for a 
given $k$.
The number of facets for all $k \leqslant 6$ is presented in 
Table~\ref{tab:nrOfFacets}.

\begin{table}[hbtp]
    \caption{The number of facets of $\STAB^{2}(\GZeroi{k})$ for $k \in 
            \{2,3,4,5,6\}$}
    \begin{center}
        \renewcommand{\arraystretch}{1.25} 
        \begin{tabular}{|l|r|r|r|r|r|}
            \hline
            $k$ & \multicolumn{1}{c|}{$2$} & \multicolumn{1}{c|}{$3$} & 
            \multicolumn{1}{c|}{$4$} & \multicolumn{1}{c|}{$5$} & 
            \multicolumn{1}{c|}{$6$}\\
            \hline
            $\#$ facets of $\STAB^{2}(\GZeroi{k})$ & 4& 16& 56& 368 & 116764\\
            \hline
        \end{tabular}
    \end{center}
    \renewcommand{\arraystretch}{1} 
    \label{tab:nrOfFacets}
\end{table}

Now we briefly present the inequalities that 
represent facets of $\STAB^{2}(\GZeroi{k})$ for $k \in \{2,3\}$. 
The ESC for a subgraph $\GI$  of order $\kI = 2$ 
with $I=\{i,j\}$ is equivalent to
\begin{subequations}
    \label{ineq:ESC2}
    \begin{align}
    0 &\leqslant X_{i,j}\label{ineq:facetGreater0}\\
    X_{i,j} & \leqslant  X_{i,i} \label{ineq:facetDiagGreaterOffDiag}\\
    X_{i,j} & \leqslant X_{j,j} \label{ineq:facetDiagGreaterOffDiag2}\\
    X_{i,i} + X_{j,j} & \leqslant 1 + X_{i,j}. 
    \label{ineq:facetSumDiagSmaller1PlusOffDiag}
    \end{align}
\end{subequations}

For a subgraph $\GI$ 
of order $\kI= 3$ with $I=\{i,j,k\}$ the ESCs is 
equivalent to~\eqref{ineq:ESC2} for all three sets $\{i,j\}$, 
$\{i,\ell\}$ and 
$\{j,\ell\}$ and the following inequalities 
\begin{subequations}
    \begin{align}
    X_{i,j} + X_{i,\ell}  
    &\leqslant X_{i,i} +  X_{j,\ell} 
    \label{inequ:homog3_1}\\
    X_{i,j} + X_{j,\ell}  
    &\leqslant X_{j,j} +  X_{i,\ell} 
    \label{inequ:homog3_2}\\
    X_{i,\ell} + X_{j,\ell}  
    &\leqslant X_{\ell,\ell} +  X_{i,j} 
    \label{inequ:homog3_3}\\
    X_{i,i} + X_{j,j} + X_{\ell,\ell} 
    & \leqslant 1 + X_{i,j} + X_{i,\ell} + X_{j,\ell}, 
    \label{inequ:inhomog3}
    \end{align}
\end{subequations}
so $3\cdot4 + 4 = 16$ inequalities, which matches Table~\ref{tab:nrOfFacets}.
We come back to these inequalities in 
Section~\ref{sec:ESCcompareToOtherHierarchies} and 
Section~\ref{sec:CompareESHandCESH}.

To summarize, we have discussed two different options to represent ESCs, 
one as convex combination and one as inequalities that represent facets.

\subsection{Comparison to Other Hierarchies}
\label{sec:ESCcompareToOtherHierarchies}
In this section we compare the ESH for the 
stable set problem to other hierarchies, as it 
has never been done before. 

The most prominent hierarchies of relaxations for general $0$--$1$ programming 
problems are the hierarchies by Sherali and Adams~\cite{SheraliAdamsHierarchy}, 
by Lov\'asz and Schrijver~\cite{LovaszSchrijverHierarchy} and by 
Lasserre~\cite{LasserreHierarchy}. We refer to 
Laurent~\cite{LaurentComparisonHierarchies} for rigorous definitions, 
comparisons and for details of applying them 
to the stable set problem. 


In fact the Lasserre hierarchy is a 
refinement of the Sherali--Adams hierarchy which is a refinement of the SDP 
based 
Lov\'asz--Schrijver hierarchy.
All three hierarchies are exact at level $\alpha(G)$, so after at 
most $\alpha(G)$ steps $\STAB(G)$ is obtained.

Silvestri~\cite{SilvestriThesis} observed that $\zESCHG{2}$ is at least as good 
as the upper bound obtained at the first level of the SDP hierarchy of 
Lov\'asz--Schrijver. This is easy to see, because this SDP 
is~\eqref{lovaszTheta_nPlus1} with 
non-negativity constraints for $X$, and every 
$\XI \in \STAB(\GI)$ is entry-wise non-negative due 
to~\eqref{ineq:facetGreater0}.
Furthermore, Silvestri proved that the bound on the $k$-th level of the 
Lasserre hierarchy is at least as good as $\zESCHG{k}$, so the Lasserre 
hierarchy yields stronger relaxations than the ESH.

A drawback of all the above hierarchies is that the size of the SDPs 
to solve grows at each level. In particular, the SDP 
at the $k$-th level of the Lasserre hierarchy has a matrix variable with  
one 
row for each subset of $i$ vertices of the $n$ vertices for every $1 \leqslant 
i \leqslant k$. 
Therefore, the matrix variable is of order $\sum_{i = 
0}^{k}\binom{n}{i}$. For the ESH this order  
remains $n+1$ on 
each level and only 
the number of constraints increases.  

Another big advantage of the ESH over the Lasserre 
hierarchy is that it is possible to include partial information of the $k$-th 
level of the hierarchy, which was exploited by  Gaar and 
Rendl~\cite{elli-diss, GaarRendl, GaarRendlFull}. 
In the case of the Lasserre 
hierarchy one needs the whole huge matrix in order to incorporate the 
information. Due to that Gvozdenovi\'c, Laurent and 
Vallentin~\cite{GvozdenovicLaurentVallentinHierarchy} introduced a new 
hierarchy where they only consider suitable principal submatrices of the huge  
matrix.

Eventually we want to compare the ESH with other 
relaxations of $\vartheta(G)$ towards $\alpha(G)$.
Lov\'asz and Schrijver~\cite{LovaszSchrijverHierarchy} proposed to add 
inequalities that boil down to \eqref{ineq:facetGreater0}, and  inequalities of 
the form \eqref{inequ:homog3_3} and \eqref{inequ:inhomog3} whenever $\{i,j\} 
\in E$. Hence, $\zESCHG{k}$ is as least as good as this bound for all $k 
\geqslant 3$.
Furthermore, Gruber and Rendl~\cite{GruberRendl} proposed to add 
inequalities of 
the form \eqref{inequ:homog3_3} and \eqref{inequ:inhomog3} also if $\{i,j\} 
\not \in E$, hence the $k$-th level of the ESH is as least as strong as 
this relaxation for every $k\geqslant 3$.

Note that Fischer, Gruber, Rendl and 
Sotirov~\cite{FischerGruberRendlSotirovBundleMaxCutEquipartition} add triangle 
inequalities into an SDP relaxation of Max-Cut. Therefore, applying the ESH
 to the Max-Cut relaxation as it is done by in~\cite{GaarRendlFull} 
can be viewed as generalization of the approach 
in~\cite{FischerGruberRendlSotirovBundleMaxCutEquipartition}.

For a discussion of other approaches for improving a relaxation by including 
information of smaller polytopes into the relaxation see~\cite{AARW}.

\section{The Compressed Exact Subgraph Hierarchy}
\label{esc:compressedESH}

In this section we newly introduce a variant of the ESH, namely the 
	compressed ESH, which at first sight is computational favorable to the ESH, 
	as it starts from a smaller SDP formulation of the Lov\'{a}sz theta 
	function. Additionally, we compare this new hierarchy to the ESH and to 
	other hierarchies from the literature.

\subsection{Two SDP Formulations of the Lov\'{a}sz Theta Function}
\label{sec:Compare2ThetaFunction}

The starting point of the new compressed ESH is an SDP formulation of the 
Lov\'{a}sz theta function~$\vartheta(G)$ by 
Lov\'{a}sz~\cite{LovaszStart}, namely
\begin{alignat}{3}
\label{lovaszTheta_n}
\tag{$T_n$}
\vartheta(G) =
& \max \quad & \langle \allones{n\times n}, X  \rangle \\ \nonumber
& \st &  \trace(X) &= 1\\ \nonumber
&& X_{i,j} &= 0 && \forall \{i,j\} \in E \\ \nonumber                  
&& X &\succcurlyeq 0\\ \nonumber
&& X &\in \Sym{n}.
\end{alignat}
As the feasible region of~\eqref{lovaszTheta_n} will be used later, we define
\begin{equation*}
\CompressedThetaBody^{2}(G) = \left\{ X \in \Sym{n} : 
\trace(X) = 1, ~  X_{i,j}=0 \quad \forall \{i,j\} \in E,~
X 
\succcurlyeq 0 
\right\}.
\end{equation*}

Before we continue, we compare the two SDP 
formulations~\eqref{lovaszTheta_nPlus1} and~\eqref{lovaszTheta_n} 
of~$\vartheta(G)$.
As already 
mentioned~\eqref{lovaszTheta_nPlus1} is an SDP with a matrix variable of 
order 
$n+1$ and $n+m+1$ equality constraints.
The formulation~\eqref{lovaszTheta_n} has a matrix variable of order $n$ 
and $m+1$ constraints, so both the number of variables and 
constraints is smaller.
Hence, in computations~\eqref{lovaszTheta_n} seems favorable.

So far, there has been a lot of work on comparing~\eqref{lovaszTheta_nPlus1} 
and~\eqref{lovaszTheta_n}. 
Gruber and Rendl~\cite{GruberRendl} showed the following. If $(x^\ast,X^\ast)$ 
is a 
feasible solution of~\eqref{lovaszTheta_nPlus1}, then 
$X' = \frac{1}{\trace(X^\ast)}X^\ast$ 
is a feasible solution of~\eqref{lovaszTheta_n} which has at least the same 
objective function value. 
Hence, an optimal solution
of~\eqref{lovaszTheta_nPlus1} can be 
transformed into an optimal solution of~\eqref{lovaszTheta_n}.
They also proved that whenever $X'$ is optimal 
for~\eqref{lovaszTheta_n}, then $X^\ast = \langle \allones{n\times n}, X'  
\rangle X'$ is optimal for~\eqref{lovaszTheta_nPlus1}.
Furthermore, Yildirim and Fan-Orzechowski\cite{YildirimFan} gave a  
transformation from a feasible solution $X'$ of~\eqref{lovaszTheta_n} to obtain 
$x^\ast$ of a 
feasible solution $(x^\ast,X^\ast)$ of~\eqref{lovaszTheta_nPlus1} with at least 
the 
same objective function value. 
Galli and 
Letchford~\cite{GalliLetchford} showed how to construct a corresponding 
$X^\ast$. 
For an optimal $X'$ the obtained optimal $(x^\ast,X^\ast)$ coincides with the 
one of 
Gruber and Rendl.
Further details can be found in~\cite{GalliLetchford}, where also the influence 
of adding certain cutting planes into~\eqref{lovaszTheta_nPlus1} 
and~\eqref{lovaszTheta_n} is discussed. 
We come back to that later 
in~Section~\ref{sec:CompareESHandCESH}.

\subsection{Introduction of the Compressed Exact Subgraph Hierarchy}
\label{sec:CESH}

Next, we newly introduce the compressed exact subgraph 
hierarchy, a 
hierarchy similar to the ESH, but it 
starts from~\eqref{lovaszTheta_n} instead of starting 
from~\eqref{lovaszTheta_nPlus1}. 
First, we verify that it makes sense to build such 
a hierarchy.

\begin{lemma}
    \label{lem:SdpWithESCOfSizeNPlus1IsAlphan}
    If we add the constraint $X \in \STAB^{2}(G)$ into~\eqref{lovaszTheta_n} 
    for a graph $G$, then 
    the optimal objective function value is $\alpha(G)$, so
    \begin{equation}
    \label{eq:SdpWithESCOfSizeNIsAlphan}
    \alpha(G) =
    \max \left\{
    \langle \allones{ n\times n},  X  \rangle :
    X \in \CompressedThetaBody^{2}(G),~
    X \in \STAB^{2}(G)
    \right\}.
    \end{equation}
\end{lemma}
\begin{proof}
	Let $(P^\mathcal{C})$ be  the SDP on the 
right-hand side of~\eqref{eq:SdpWithESCOfSizeNIsAlphan}, 
let $z^\mathcal{C}$ be its optimal objective function value 
and let $\mathcal{S}(G) = \{s_{1}, \dots, s_{t}\}$. 

Let without loss of generality $s_t$ be the incidence vector of a maximum 
stable set of $G$, and $s_1$ be the incidence vector of the empty set, 
which is of course stable. Then clearly $X = 
\frac{1}{\alpha(G)}s_ts_t^{\transposedVec} + \left( 1 - 
\frac{1}{\alpha(G)}\right)s_1s_1^{\transposedVec}$ 
is feasible for $(P_{\alpha}^\mathcal{C})$ and has objective function value 
$\alpha(G)$, 
so $\alpha(G) \leqslant z^\mathcal{C}$ holds.

	Furthermore, any feasible solution $X$ of $(P^{\mathcal{C}})$ can be 
written as 
\begin{equation*}
X = \sum_{i=1}^{t} \lambda_{i}s_{i}s_{i}^{\transposedVec}
\end{equation*}
for some $\lambda \in \simplex_{t}$ because $X \in \STAB^{2}(G)$ holds, and it 
fulfills
\begin{equation*}
1 = \trace(X) 
= \sum_{i=1}^{t} \lambda_{i}\trace(s_{i}s_{i}^{\transposedVec}) 
= \sum_{i=1}^{t} \lambda_{i}\allones{n}^{\transposedVec}s_{i}.
\end{equation*}
In consequence, the objective function value of $X$ for $(P^{\mathcal{C}})$ is 
equal to
\begin{equation*}
\langle \allones{n\times n}, X  \rangle
= \sum_{i=1}^{t} \lambda_{i} \langle \allones{n\times n}, 
s_{i}s_{i}^{\transposedVec} \rangle
= \sum_{i=1}^{t} \lambda_{i} (\allones{n}^{\transposedVec}s_{i})^2
\leqslant  
\alpha(G)\sum_{i=1}^{t} \lambda_{i} \allones{n}^{\transposedVec}s_{i}
=
\alpha(G)
\end{equation*}
and hence $z^\mathcal{C} \leqslant \alpha(G)$ holds, which finishes the proof.
\end{proof}

Lemma~\ref{lem:SdpWithESCOfSizeNPlus1IsAlphan}
corresponds to Lemma~\ref{lem:SdpWithESCOfSizeNIsAlpha} for the ESH and
justifies the introduction of the compressed exact subgraph hierarchy.
\begin{definition}
    Let $G = (V,E)$ be a graph with $|V| = n$ and
    let $J$ be a set of subsets of $V$. Then $\zCESCJ$ is the optimal objective 
    function 
    of~\eqref{lovaszTheta_n} with the ESC for every 
    subgraph induced by a set in $J$, so
    \begin{equation}
    \label{SDPRelaxationWithCESCHierarchyJ}
    \zCESCJ = 
    \max \left\{
    \langle \allones{n\times n}, X  \rangle:
    X \in \CompressedThetaBody^2(G),~
    \XI \in \STAB^{2}(\GI) \quad \forall I \in J
    \right\}.
    \end{equation}    
    
    For $k \in \Nzero$ with $k \leqslant n$
    the $k$-th level of the compressed exact subgraph hierarchy (CESH) 
    is defined as
    $\zCESCHG{k} = \zCESCJk$.
\end{definition}

As in the case of the ESH we can deduce the 
following result for the CESH.
\begin{lemma}
    \label{lem:PropertiesCESH}
    Let $G = (V,E)$ be a graph with $|V| = n$. 
    Then
    \begin{equation*}
    \vartheta(G) = \zCESCHG{0} =  \zCESCHG{1} \geqslant \zCESCHG{k-1} \geqslant 
    \zCESCHG{k} \geqslant \zCESCHG{n} = \alpha(G)
    \end{equation*}
    holds for all $k \in \{1, \dots, n\}$.
\end{lemma}
\begin{proof}
    Analogous to the proof of Lemma~\ref{lem:PropertiesESH}.
\end{proof}

Hence, due to Lemma~\ref{lem:PropertiesESH} and 
Lemma~\ref{lem:PropertiesCESH} 
both 
the ESH and the CESH start at~$\vartheta(G)$ at level $1$ and reach $\alpha(G)$ 
on 
level $n$.

\subsection{Comparison to Other Hierarchies}
Before we continue to consider the differences between the ESH and the 
CESH, we compare the CESH with other 
relaxations of $\alpha(G)$ based on~\eqref{lovaszTheta_n}.

Schrijver~\cite{SchrijverAddNonNeg} suggested to add non-negativity constraints 
into~\eqref{lovaszTheta_n} to obtain stronger bounds.
Galli and Letchford~\cite{GalliLetchford} proved that it is equivalent to 
include non-negativity 
constraints into~\eqref{lovaszTheta_nPlus1} and~\eqref{lovaszTheta_n}, 
so $\zESCHG{2}$ is a stronger bound than this one because it induces 
non-negativity in~\eqref{lovaszTheta_nPlus1}.
Lemma~\ref{lem:XIInSTAB2GIIffXIInSTAB2EmptyGraph} implies 
that also for~\eqref{lovaszTheta_n} it is equivalent to include 
$\XI \in \STAB^{2}(\GI)$ and  
$\XI \in \STAB^{2}(\GZeroi{\kI})$, 
so $\zCESCHG{2}$ induces non-negativity due 
to~\eqref{ineq:facetGreater0}.
Hence, also $\zCESCHG{2}$ is as least as good as the bound of 
Schrijver.

Dukanovic and Rendl~\cite{DukanovicRendl} proposed to add so-called 
triangle inequalities to~\eqref{lovaszTheta_n}. 
Silvestri~\cite{SilvestriThesis} showed that $\zCESCHG{3}$ is at least as good 
as upper bound as the 
bound of Dukanovic and Rendl. 
This is intuitive, because the triangle 
inequalities correspond to~\eqref{inequ:homog3_1},~\eqref{inequ:homog3_2} 
and~\eqref{inequ:homog3_3} and therefore represent faces of $\STAB^{2}(\GI)$ 
for $\kI = 3$.
As a result, the CESH can be seen as a 
generalization of the relaxation of~\cite{DukanovicRendl}.

\subsection{Comparison of the CESH and the ESH}
\label{sec:CompareESHandCESH}
Now we continue our comparison of the bounds 
based on 
the ESH and our new CESH. 

\begin{theorem}
    \label{lem:CompareBoundsJ}
    Let $G = (V,E)$ be a graph with $|V| = n$ and 
    let $J$ be a set of subsets of $V$.
    Then $\zESCJ \leqslant \zCESCJ$.
\end{theorem}
\begin{proof}
    We consider the transformation of an 
    optimal solution of~\eqref{lovaszTheta_nPlus1} into an optimal 
    solution of~\eqref{lovaszTheta_n} by Gruber and 
    Rendl~\cite{GruberRendl}.
    We show that this transformation applied to the optimal solution 
    of~\eqref{SDPRelaxationWithESCHierarchyJ} yields a feasible solution 
    of~\eqref{SDPRelaxationWithCESCHierarchyJ} with at least the same objective 
    function 
    value, thus $\zESCJ \leqslant \zCESCJ$ holds.
    
    Towards that end, let $(\xOpt,\XOpt)$ be an optimal solution 
    of~\eqref{SDPRelaxationWithESCHierarchyJ} and
     $\gamma = \zESCJ = \allones{n}^{\transposedVec}\xOpt$ its objective 
    function value.
    Let $\XOptC = \frac{1}{\gamma}\XOpt$.
    
    First, we show that $\XOptC$ is feasible 
    for~\eqref{SDPRelaxationWithCESCHierarchyJ}.
    Clearly 
    $\XOpt -\xOpt (\xOpt)^{\transposedVec} \succcurlyeq 0$ 
    and $\gamma \geqslant 0$ imply
    $\XOptC \succcurlyeq 0$.
    Furthermore, due to $\XOpt_{i,j} = 0$ for all  $\{i,j\} \in E$ we have 
    $\XOptC_{i,j} = 0$ for all $\{i,j\} \in E$.
    Additional to that
    \begin{equation*}
    \trace(\XOptC) = \frac{1}{\gamma}\trace(\XOpt)
    = \frac{1}{\gamma} \allones{n}^{\transposedVec}\xOpt 
    = \frac{1}{\gamma}\gamma
    = 1,
    \end{equation*}
    so $\XOptC$ is feasible for~\eqref{lovaszTheta_n}.
    
    What is left to check for feasibility are the ESCs.
    We can rewrite $\XOptI \in \STAB^2(\GI)$ as 
    $\XOptI = \sum_{i=1}^{\tI} \lambdaIi{i}\SIi$ 
    for $\sum_{i=1}^{\tI}\lambdaIi{i} = 1$ and $\lambdaIi{i} \geqslant 0$ for 
    all $1 \leqslant i \leqslant \tI$. 
    Let w.l.o.g.\ $\SIiIndex{1}$ be the zero matrix of 
    dimension $\kI \times \kI$, i.e. the first stable set matrix corresponds to 
    the empty set.
    Then we define
    \begin{equation*}
    \lambdaICi{i} = \left\{\begin{array}{ll}
    \frac{1}{\gamma} \lambdaIi{i} & \text{for } 2 \leqslant i \leqslant \tI\\
     \frac{1}{\gamma} \lambdaIi{i} +  \frac{\gamma-1}{\gamma} 
     & \text{for } i = 1.
    \end{array}\right.
    \end{equation*}
    It is easy to see that $ \lambdaICi{i} \geqslant 0$ for all 
    $1 \leqslant i \leqslant \tI$ and that 
    \begin{equation*}
    \sum_{i=1}^{\tI}\lambdaICi{i} 
    = \frac{1}{\gamma} \lambdaIi{1} +  \frac{\gamma-1}{\gamma} 
     + \frac{1}{\gamma}\sum_{i=2}^{\tI}\lambdaIi{i} 
    = \frac{1}{\gamma} +   \frac{\gamma-1}{\gamma} 
    = 1
    \end{equation*}
    holds.
    Furthermore, because $\SIiIndex{1}$ is a zero matrix and so 
    $\frac{\gamma-1}{\gamma} \SIiIndex{1} = 0$, we have 
    \begin{equation*}
    \XOptCI = \frac{1}{\gamma}\XOptI
    = \sum_{i=1}^{\tI} \frac{1}{\gamma}\lambdaIi{i}\SIi
    = \left(\frac{1}{\gamma}\lambdaIi{i} +  \frac{\gamma-1}{\gamma} 
    \right)\SIiIndex{1}
    + \sum_{i=2}^{\tI} \lambdaICi{i}\SIi
    = \sum_{i=1}^{\tI} \lambdaICi{i}\SIi.
    \end{equation*}
    As a consequence $\XOptCI \in \STAB^2(\GI)$ and thus $\XOptCI$ is feasible 
    for~\eqref{SDPRelaxationWithCESCHierarchyJ}.
    
    It remains to determine the objective function value of $\XOptCI$ 
    for~\eqref{SDPRelaxationWithCESCHierarchyJ}.
    From $\XOpt -\xOpt (\xOpt)^{\transposedVec} \succcurlyeq 0$ it follows that
    $\allones{n}^{\transposedVec}
    (\XOpt -\xOpt (\xOpt)^{\transposedVec})\allones{n} \geqslant 0$
    and hence
    $\langle \allones{n\times n}, 
    \XOpt -\xOpt (\xOpt)^{\transposedVec} \rangle \geqslant 0$.
    This implies that
\begin{equation*}
    \langle \allones{n\times n}, \XOpt \rangle \geqslant
    \langle  \allones{n\times n}, \xOpt (\xOpt)^{\transposedVec} \rangle 
    = \allones{n}^{\transposedVec}\xOpt (\xOpt)^{\transposedVec}\allones{n}
    = (\allones{n}^{\transposedVec}\xOpt)^2 
    = \gamma^2
\end{equation*}
    holds, thus 
\begin{equation*}
    \langle \allones{n\times n}, \XOptC  \rangle
    = \frac{1}{\gamma}\langle \allones{n\times n}, \XOpt  \rangle
    \geqslant \frac{1}{\gamma}\gamma^2 = \gamma.
\end{equation*}
    
    To summarize, $\XOptC$ is a feasible solution 
    of~\eqref{SDPRelaxationWithCESCHierarchyJ} with objective function value
    $\gamma = \zESCJ$. Therefore, the optimal objective function value of 
    the 
    maximization problem~\eqref{SDPRelaxationWithCESCHierarchyJ} is at least 
    $\zESCJ$, so $\zESCJ \leqslant \zCESCJ$.
\end{proof}

Theorem~\ref{lem:CompareBoundsJ} states that the bounds obtained by 
starting 
from~\eqref{lovaszTheta_nPlus1} and including some ESCs is always at least 
as good as the bound obtained by starting from~\eqref{lovaszTheta_n} and 
including the same ESCs. 
In particular, this implies that the relaxation on the $k$-th level of the 
ESH 
is at least as good as the relaxation on the $k$-th level of the CESH,
which is formalized in the following corollary.
\begin{corollary}
    \label{lem:CompareBoundsH}
    Let $G = (V,E)$ be a graph with $|V| = n$ and let $k \in \Nzero$, $k 
    \leqslant 
    n$. 
    Then $\zESCHG{k} \leqslant \zCESCHG{k}$.
\end{corollary}

We now further investigate the theoretical difference between the ESH 
and the CESH,
especially in the light of the results of Galli and 
Letchford~\cite{GalliLetchford}.
They proved that whenever a collection of  
homogeneous inequalities is added to~\eqref{lovaszTheta_nPlus1}, the resulting 
optimal solution yields a feasible solution for~\eqref{lovaszTheta_n} with the 
same collection of 
inequalities, which has at least the same objective function value. 
This implies that adding homogeneous inequalities to~\eqref{lovaszTheta_nPlus1} 
gives stronger bounds on $\alpha(G)$ than adding the same inequalities 
to~\eqref{lovaszTheta_n}.

If we consider the ESCs in more detail as we did in 
Section~\ref{sec:ESCAsInequalities}, then in turns out 
that for 
$k=2$ the 
inequalities~\eqref{ineq:facetGreater0},~\eqref{ineq:facetDiagGreaterOffDiag} 
and~\eqref{ineq:facetDiagGreaterOffDiag2}
are homogeneous, while~\eqref{ineq:facetSumDiagSmaller1PlusOffDiag} is 
inhomogeneous, so inhomogeneous inequalities are needed to represent ESCs.

Next, we give an intuition for 
the different behavior of inhomogeneous inequalities 
for the two SDP formulations of the Lov\'{a}sz 
theta 
function~\eqref{lovaszTheta_nPlus1} and~\eqref{lovaszTheta_n}.
Let $(\xOpt,\XOpt)$ be an optimal solution 
of~\eqref{lovaszTheta_nPlus1} with additional 
constraints~\eqref{ineq:ESC2}. 
From the proof of Lemma~\ref{lem:CompareBoundsJ} we know that 
$\XOptC = \frac{1}{\gamma}\XOpt$ 
is a feasible solution  
of~\eqref{lovaszTheta_n} with additional constraints~\eqref{ineq:ESC2}. 
Indeed, the homogeneous 
inequalities~\eqref{ineq:facetGreater0},~\eqref{ineq:facetDiagGreaterOffDiag} 
and~\eqref{ineq:facetDiagGreaterOffDiag2} are preserved under scaling,  
matching~\cite{GalliLetchford}.
Scaling~\eqref{ineq:facetSumDiagSmaller1PlusOffDiag} 
with~$\frac{1}{\gamma}$ yields that $\XOptC$ satisfies
\begin{equation*}
\XOptC_{i,i} + \XOptC_{j,j}  \leqslant \frac{1}{\gamma} 
+ \XOptC_{i,j} 
\end{equation*}
and since $\frac{1}{\gamma} \leqslant 1$ it follows that $\XOptC$ 
satisfies~\eqref{ineq:facetSumDiagSmaller1PlusOffDiag}.

If $\XOptC$ is an optimal solution 
of~\eqref{lovaszTheta_n} with additional 
constraints~\eqref{ineq:ESC2}
and we use the transformation $\XOpt =\gamma\XOptC$, 
then clearly $\XOpt$ 
satisfies~\eqref{ineq:facetGreater0},~\eqref{ineq:facetDiagGreaterOffDiag} 
and~\eqref{ineq:facetDiagGreaterOffDiag2}. 
Scaling~\eqref{ineq:facetSumDiagSmaller1PlusOffDiag} with $\gamma$ 
yields that 
\begin{equation*}
\XOpt_{i,i} + \XOpt_{j,j}  \leqslant \gamma 
+ \XOpt_{i,j} 
\end{equation*}
holds for $\XOpt$. This does not imply that $\XOpt$ 
fulfills~\eqref{ineq:facetSumDiagSmaller1PlusOffDiag}
as $\gamma \geqslant 1$.

To summarize, this consideration confirms that the ESCs for $\kI = 2$ 
yield a 
stronger restriction in~\eqref{lovaszTheta_nPlus1} than 
they do in~\eqref{lovaszTheta_n}. 
This gap of the bounds gets even larger for larger $\kI$, so 
for example for $\kI=3$ 
the  inequality~\eqref{inequ:inhomog3} is inhomogeneous.
This concludes our investigation of the new CESH.

\section{The Scaled Exact Subgraph Hierarchy}
\label{sec:scaledESC}

In Section~\ref{esc:compressedESH} we saw that including an ESC 
	into~\eqref{lovaszTheta_nPlus1} as in the ESH gives a stronger bound than 
	including the 
	same 
	ESC into~\eqref{lovaszTheta_n} as in the CESH.
	In this section we investigate whether this is due to a suboptimal 
	definition of the ESCs for the later case. 
	In particular, we go back to the intuition behind 
	ESCs for~\eqref{lovaszTheta_nPlus1} and transfer this intuition 
	to~\eqref{lovaszTheta_n}. This will lead to the new definitions of scaled 
	ESCs
	and the scaled ESH. We will explore this hierarchy and compare the CESH and 
	the scaled ESH in detail.

\subsection{Introduction of the Scaled Exact Subgraph Hierarchy}

To start, observe the following.
It can be confirmed easily that both~\eqref{lovaszTheta_nPlus1} 
and~\eqref{lovaszTheta_n} are upper bounds on 
$\alpha(G)$. 
Let $s \in \mathcal{S}(G)$ be a stable 
set vector that corresponds to a maximum stable set. 
Then $X^\ast = ss^T$ is feasible 
for~\eqref{lovaszTheta_nPlus1} and has 
objective function value $\alpha(G)$.
Therefore, intuitively $\STAB^2(G)$ defines exactly the 
appropriate polytope for~\eqref{lovaszTheta_nPlus1}.

For~\eqref{lovaszTheta_n} the matrix $X' = \frac{1}{s^{T}s}ss^T$ yields 
a feasible solution with objective function value $\alpha(G)$, whereas 
$X^\ast=ss^T$ is not feasible unless $\alpha(G)=1$.
Hence, intuitively it makes more sense to consider the 
polytope spanned by matrices of the form $\frac{1}{s^{T}s}ss^T$ for $s \in 
\mathcal{S}(G)$ for~\eqref{lovaszTheta_n} than to consider $\STAB^2(G)$.
This leads to the following definition.

\begin{definition}
    \label{def:SSTAB2}
    Let $G = (V,E)$ be a graph with $|V| = n$. Then the scaled squared stable 
    set polytope 
    $\SSTAB^{2}(G)$ of $G$ is defined as 
    \begin{equation*}
    \SSTAB^{2}(G) = \conv\left(
    \left\{\frac{1}{s^{T}s}ss^{\transposedVec}: s \in 
    \mathcal{S}(G),~ 
    s \neq \allzeros{n}\right\}
    \cup \{\allzeros{n\times n}\} \right).
    \end{equation*}  
\end{definition}

The goal of this section is to investigate a new modified version of 
the CESH based on the scaled squared stable set polytope defined in the 
following way.

\begin{definition}
    Let $G = (V,E)$ be a graph and let $I \subseteq V$.
    Then the scaled exact subgraph constraint (SESC) for $\GI$ is defined as
    $
    \XI \in \SSTAB^{2}(\GI).
    $
    Furthermore, let $|V| = n$ and
    let $J$ be a set of subsets of $V$. Then $\zSESCJ$ is the optimal 
    objective function value 
    of~\eqref{lovaszTheta_n} with the SESC for every 
    subgraph induced by a set in $J$, so
    \begin{equation}
    \label{SDPRelaxationWithSESCHierarchyJ}
    \zSESCJ = 
    \max \left\{
    \langle \allones{n\times n}, X  \rangle: 
    X \in \CompressedThetaBody^2(G),~
    \XI \in \SSTAB^{2}(\GI) \quad \forall I \in J
    \right\}.
    \end{equation}
    
    For $k \in \Nzero$ with 
    $k \leqslant n$
    the $k$-th level of the scaled exact subgraph hierarchy (SESH) 
    is defined as
    $\zSESCHG{k} = \zSESCJk$.
\end{definition}

Note that with the considerations above  it 
does not make sense to 
include the SESC for the whole graph $G$ into~\eqref{lovaszTheta_nPlus1}, as 
this SDP does 
not yield an upper bound on $\alpha(G)$, because all solutions 
corresponding to $\alpha(G)$ are not feasible.
Hence, we introduce a hierarchy based on SESCs only starting 
from~\eqref{lovaszTheta_n} and not from~\eqref{lovaszTheta_nPlus1}.

Additionally, note that a priory we do not know whether the SESH has as 
nice properties as the ESH and the CESH.

\subsection{Comparison of the SESH and the CESH}
The next lemma 
is the key 
ingredient to 
compare the SESH to the CESH.

\begin{lemma}
    \label{lem:XSSTABiffXSTABandTraceLeq1}
Let $G = (V,E)$ be a graph. Then
$X \in \SSTAB^2(G)$ holds
if and only if 
$X \in \STAB^2(G)$  and  $\trace(X) \leqslant 1$.
\end{lemma}
\begin{proof}
Let $\mathcal{S}(G) = \{s_{1}, \dots, s_{t}\}$ and let w.l.o.g.\ $s_{1}= 
\allzeros{n}$, i.e.\ the first stable set is the empty set.

If $X \in \SSTAB^2(G)$, then $X$ can be written as
\begin{equation*}
X = \lambdaS_{1}\allzeros{n\times n} + \sum_{i=2}^{t} 
\lambdaS_{i}\frac{1}{s_{i}^{\transposedVec}s_{i}}s_{i}s_{i}^{\transposedVec}
\end{equation*}
for some $\lambdaS \in \simplex_t$.
It is easy to see that 
\begin{equation*}
\trace(X) 
= \lambdaS_{1}\trace(\allzeros{n\times n}) + \sum_{i=2}^{t} 
\lambdaS_{i}\frac{1}{s_{i}^{\transposedVec}s_{i}}\trace(s_{i}s_{i}^{\transposedVec})
= \sum_{i=2}^{t} \lambdaS_{i} \leqslant 1
\end{equation*}
holds. We define 
$\lambdaN_{i} = \lambdaS_{i}\frac{1}{s_{i}^{\transposedVec}s_{i}}$ for 
$2 \leqslant i \leqslant t$. Then clearly 
$\lambdaS_{i} 
\geqslant \lambdaS_{i}\frac{1}{s_{i}^{\transposedVec}s_{i}} 
=  \lambda_{i} \geqslant 0$ 
holds because $s_{i}^{\transposedVec}s_{i} \geqslant 1$ for all 
$2 \leqslant i \leqslant t$.
Let $\lambdaN_{1} 
= 1 - \sum_{i=2}^{t} \lambdaN_{i}$, then $\lambdaN_{1} \geqslant 1 -  
\sum_{i=2}^{t} \lambdaS_{i} = 
\lambdaS_{1} \geqslant 0$ holds.
Hence
\begin{equation*}
X = \lambdaN_{1}\allzeros{n\times n} + \sum_{i=2}^{t} 
\lambdaN_{i}s_{i}s_{i}^{\transposedVec}
\end{equation*}
for $\lambdaN \in \simplex_t$ and therefore $X \in \STAB^2(G)$. 
Hence, $X \in \SSTAB^2(G)$ implies that $X \in \STAB^2(G)$ and 
$\trace(X) \leqslant 1$ holds.

Now assume $X \in \STAB^2(G)$ and $\trace(X) \leqslant 1$.
Then $X$ can be rewritten as
\begin{equation*}
X = \lambdaN_{1}\allzeros{n\times n} + \sum_{i=2}^{t} 
\lambdaN_{i}s_{i}s_{i}^{\transposedVec}
\end{equation*}
for some $\lambdaN \in \simplex_t$.
Then, because $\trace(s_{i}s_{i}^{\transposedVec}) = 
s_{i}^{\transposedVec}s_{i}$, we have  
\begin{equation}
\label{eq:traceLeq1}
1 \geqslant \trace(X) 
=  \lambdaN_{1}\trace(\allzeros{n\times n}) + \sum_{i=2}^{t} 
\lambdaN_{i}\trace(s_{i}s_{i}^{\transposedVec}) 
= \sum_{i=2}^{t} 
\lambdaN_{i}s_{i}^{\transposedVec}s_{i}.
\end{equation}
We define $\lambdaS_{i} = \lambdaN_{i}s_{i}^{\transposedVec}s_{i}$ for $2 
\leqslant i \leqslant t$ and 
$\lambdaS_{1} = 1 - \sum_{i=2}^{t} \lambdaS_{i}$. 
Then clearly $\lambdaS_{i} \geqslant 0$ holds for $2 
\leqslant i \leqslant t$. Furthermore, \eqref{eq:traceLeq1} implies that 
$\lambdaS_{1} \geqslant 0$ holds, so $\lambdaS \in \simplex_{t}$. 
This together with 
\begin{equation*}
X = \lambdaS_{1}\allzeros{n\times n} + \sum_{i=2}^{t} 
\lambdaS_{i} \frac{1}{s_{i}^{\transposedVec}s_{i}} s_{i}s_{i}^{\transposedVec}
\end{equation*}
implies that $X \in \SSTAB^2(G)$.
\end{proof}

Lemma~\ref{lem:XSSTABiffXSTABandTraceLeq1} 
allows us to prove the following.
\begin{theorem}
    \label{thm:CompareBoundsJScaled}
    Let $G = (V,E)$ be a graph and 
    let $J$ be a set of subsets of $V$.
    Then $\zSESCJ = \zCESCJ$.
    In particular, $\zSESCHG{k} = \zCESCHG{k}$. 
\end{theorem}
\begin{proof}
    Due to Lemma~\ref{lem:XSSTABiffXSTABandTraceLeq1} the SESC 
    $\XI \in \SSTAB^2(\GI)$ in  $\zSESCJ$ 
    can be replaced by the ESC $\XI \in \STAB^2(\GI)$ and 
    $\trace(\XI) \leqslant 1$. 
    The latter is redundant,    
    as $\trace(X) = 1$ is fulfilled by all $X \in \CompressedThetaBody^{2}(G)$ 
    and all elements on the main diagonal of $X$ are non-negative because 
    $X \succcurlyeq 0$.
    Thus, $\zSESCJ = \zCESCJ$ and $\zSESCHG{k} = \zCESCHG{k}$ hold.
\end{proof}

Theorem~\ref{thm:CompareBoundsJScaled} implies that the SESH and the CESH 
coincide 
and in particular that the SESH has the same properties as the CESH stated in 
Lemma~\ref{lem:PropertiesESH}, which we now forumlate explicitly.

\begin{corollary}
    \label{cor:PropertiesSESH}
    Let $G = (V,E)$ be a graph with $|V| = n$. 
    Then
    \begin{equation*}
    \vartheta(G) = \zSESCHG{0} =  \zSESCHG{1} \geqslant \zSESCHG{k-1} \geqslant 
    \zSESCHG{k} \geqslant \zSESCHG{n} = \alpha(G)
    \end{equation*}
    holds for all $k \in \{1, \dots, n\}$.
\end{corollary}

Hence, even though intuitively it makes more sense to add SESCs 
into~\eqref{lovaszTheta_n} instead of ESCs, 
both versions give the same bound and the SESH and the CESH coincide.

\section{Computational Comparison}
\label{sec:computations}

In the previous sections we have theoretically investigated first the 
original ESH, which starts from~\eqref{lovaszTheta_nPlus1} and includes 
ESCs. 
Next, we introduced the CESH, which starts from~\eqref{lovaszTheta_n} and 
includes ESCs 
and finally the SESH which starts from~\eqref{lovaszTheta_n} and includes SESCs.
Each of these hierarchies can be exploited computationally by including a 
wisely chosen subset $J$ of all possible ESCs or SESCs. 
We denote the 
resulting bounds based on the ESH, the CESH and the SESH by $\zESCJ$, 
$\zCESCJ$ and $\zSESCJ$, respectively. 
So far we 
have proven in Lemma~\ref{lem:CompareBoundsJ} and 
Theorem~\ref{thm:CompareBoundsJScaled}
that $\zSESCJ = \zCESCJ \geqslant \zESCJ$ holds for all 
graphs $G$ and for all set of subsets $J$, 
hence the bounds based on the CESH and the SESH coincide and the bounds 
based on the ESH are always as least as good as those bounds.

In this section we compare the ESH and the CESH computationally. 
We refrain from computations with 
SESH since both the obtained bounds and the sizes of the SDPs are the 
same for SESH and CESH.
First, we are interested in whether $\zESCJ$ 
is significantly 
better than $\zCESCJ$.
Second, we are interested in the running times. In theory, the running 
times for $\zCESCJ$ should be smaller, because the matrix variable is 
of order $n$ instead of $n+1$ and the number of equality constraints is $n$ 
less.

We consider several graphs in various settings.
Some graphs are from the Erd\H{o}s-R\'enyi model $G(n,p)$ for 
different values of $n$ and $p$ (the probability that 
an edge is present in the graph), 
some are complement graphs of graphs of the second DIMACS implementation 
challenge~\cite{DIMACS1992} 
and some come from the house of graphs 
collection~\cite{HouseOfGraphs}. 
Furthermore, there is a spin glass graph 
(see~\cite{FischerGruberRendlSotirovBundleMaxCutEquipartition}), 
a Paley graph, a circulant and a cubic graph among the 
instances.
In the computations we always compare including all ESCs 
of the same set $J$ into~\eqref{lovaszTheta_nPlus1} and~\eqref{lovaszTheta_n}, 
so we compute $\zESCJ$ and $\zCESCJ$. 
The source code and all the used graphs are available online at 
	\url{https://arxiv.org/src/2003.13605/anc}.

All computations are done on an Intel(R) Core(TM) i7-7700 CPU @ 3.60GHz 
with 32 GB RAM with MATLAB. 
We use the interior point solver MOSEK~\cite{mosek} for solving the SDPs.
Note that there is a lot of research on how to solve SDPs of the 
form~\eqref{SDPRelaxationWithESCHierarchyJ} much faster using the bundle 
method, see Gaar~\cite{elli-diss} and Gaar and 
Rendl~\cite{GaarRendl,GaarRendlFull}.
We refrain from using these involved methods, as we are 
interested in comparing the bounds in a simple way.

In the first experiment, we compare levels of the ESH and the CESH. For 
including all possible ESCs of order $k$ 
into a graph of order $n$ we have $\binom{n}{k}$ additional ESCs to the 
SDPs~\eqref{lovaszTheta_nPlus1} 
and~\eqref{lovaszTheta_n}, so these computations are out of reach rather 
quickly.
Table~\ref{tab:ttfallESC} summarizes the values of $\zCESCJ$ and $\zESCJ$ for 
including all ESCs for $k \in  \{0,2,3,4\}$ and presents the running times in 
seconds to solve the 
corresponding SDPs.

First, we note that indeed the computation of~$\vartheta(G)$ 
(corresponds to the column $k=0$) yields the same value for  
computing it via~\eqref{lovaszTheta_nPlus1} and~\eqref{lovaszTheta_n}. 
Furthermore, the 
computations confirm that $\zESCJ \leqslant \zCESCJ$ holds for all graphs~$G$. 
On the second level 
of the ESH and the CESH the two values coincide for almost all graphs. Only the 
instances HoG\_34272, HoG\_34274 and HoG\_34276 show a significant difference. 
On the third and fourth level the difference is more substantial. 
This is not surprising, as there are more inhomogeneous facets defining 
$\STAB^2$ in these cases. In the running times there is almost no difference 
for small graphs with not so many ESCs. Only if the number of ESCs becomes 
larger, typically the computation time for $\zCESCJ$ is significantly shorter. 
However, most of the times this comes with a worse bound.

%
%

\begin{table}[ht]
\caption{The values of $\zESCJ$ and $\zCESCJ$ for different graphs $G$ with  
including all ESCs of order $0$ (corresponds to $\vartheta(G)$), $2$, $3$ and 
$4$ and the running times to compute the values}
\label{tab:ttfallESC} 
\begin{center}
    \begin{small}
        \hspace*{-1cm}   
    \setlength{\tabcolsep}{3pt}
\begin{tabular}{|l|r|r|rrrr|rrrr|}
\hline
Name & $n$ &&\multicolumn{4}{c|}{value of $\zESCJnoG$/$\zCESCJnoG$}
&\multicolumn{4}{c|}{running time}\\
\cline{4-11}
&$m$ &  & $\vartheta(G)$ & $J = J_2$ &  
$J = J_3$ & $J = J_4$ &$\vartheta(G)$ & $J = J_2$ &  
$J = J_3$ & $J = J_4$ \\    
\hline 
HoG\_34272      &    9 &  $\zESCJnoG$&   3.3380 &  3.2729 &  3.0605 & 3.0000 
&    0.03   &    0.04   &    0.35 &    0.42   \\
&   17 & $\zCESCJnoG$&   3.3380 &  3.2763 &  3.1765 & 3.0000 &    0.03   &    
0.04   &    0.10 &    0.19   \\ \hline 
HoG\_15599      &   20 &  $\zESCJnoG$&   7.8202 &  7.8202 &  7.4437 & 7.0000 
&    0.05   &    0.12   &   15.63 & 2545.20   \\
&   44 & $\zCESCJnoG$&   7.8202 &  7.8202 &  7.4761 & 7.4291 &    0.04   &    
0.11   &   16.27 & 3439.99   \\ \hline 
CubicVT26\_5    &   26 &  $\zESCJnoG$&  11.8171 & 11.8171 & 10.9961 &     \_ 
&    0.04   &    0.25   &  101.19 &    \_   \\
&   39 & $\zCESCJnoG$&  11.8171 & 11.8171 & 11.0035 &        &    0.04   &    
0.22   &   99.11 &    \\ \hline
HoG\_34274      &   36 &  $\zESCJnoG$&  13.2317 & 13.0915 & 12.1661 &     \_ 
&    0.05   &    1.27   & 4026.75 &   \_  \\
&   72 & $\zCESCJnoG$&  13.2317 & 13.1052 & 12.5881 &        &    0.05   &    
1.39   & 3700.50 &    \\ \hline
HoG\_6575       &   45 &  $\zESCJnoG$&  15.0530 & 15.0530 &    \_   &     \_ 
&    0.07   &    2.18   &    \_   &   \_  \\
&  225 & $\zCESCJnoG$&  15.0530 & 15.0530 &         &        &    0.05   &    
2.18   &         &    \\ \hline
MANN\_a9        &   45 &  $\zESCJnoG$&  17.4750 & 17.4750 &    \_   &     \_ 
&    0.05   &    2.38   &    \_   &    \_  \\
&   72 & $\zCESCJnoG$&  17.4750 & 17.4750 &         &        &    0.04   &    
2.70   &         &    \\ \hline
Circulant47\_30 &   47 &  $\zESCJnoG$&  14.3022 & 14.3022 &    \_   &     \_ 
&    0.07   &    3.80   &    \_   &   \_  \\
&  282 & $\zCESCJnoG$&  14.3022 & 14.3022 &         &        &    0.06   &    
3.34   &         &    \\ \hline
G\_50\_025      &   50 &  $\zESCJnoG$&  13.5642 & 13.4554 &    \_   &     \_ 
&    0.09   &    5.71   &    \_   &    \_  \\
&  308 & $\zCESCJnoG$&  13.5642 & 13.4555 &         &        &    0.07   &    
5.36   &         &    \\ \hline
G\_60\_025      &   60 &  $\zESCJnoG$&  14.2815 & 14.2013 &    \_   &     \_ 
&    0.13   &   12.40   &    \_   &   \_  \\
&  450 & $\zCESCJnoG$&  14.2815 & 14.2013 &         &        &    0.11   &   
12.19   &         &    \\ \hline
Paley61         &   61 &  $\zESCJnoG$&   7.8102 &  7.8102 &    \_   &     \_ 
&    0.23   &    6.52   &    \_   &    \_   \\
&  915 & $\zCESCJnoG$&   7.8102 &  7.8102 &         &        &    0.17   &    
6.22   &         &    \\ \hline
hamming6\_4     &   64 &  $\zESCJnoG$&   5.3333 &  4.0000 &    \_   &     \_ 
&    0.36   &   10.60   &    \_   &    \_  \\
& 1312 & $\zCESCJnoG$&   5.3333 &  4.0000 &         &        &    0.29   &    
8.44   &         &    \\ \hline
HoG\_34276      &   72 &  $\zESCJnoG$&  26.4635 & 26.1831 &    \_   &     \_ 
&    0.06   &   30.70   &    \_   &   \_   \\
&  144 & $\zCESCJnoG$&  26.4635 & 26.2105 &         &        &    0.07   &   
32.67   &         &    \\ \hline
G\_80\_050      &   80 &  $\zESCJnoG$&   9.4353 &  9.3812 &    \_   &     \_ 
&    0.87   &   60.25   &    \_   &    \_   \\
& 1620 & $\zCESCJnoG$&   9.4353 &  9.3812 &         &        &    0.85   &   
60.63   &         &    \\ \hline
G\_100\_025     &  100 &  $\zESCJnoG$&  19.4408 & 19.2830 &    \_   &     \_ 
&    0.61   &  170.61   &    \_   &    \_   \\
& 1243 & $\zCESCJnoG$&  19.4408 & 19.2830 &         &        &    0.52   &  
178.05   &         &    \\ \hline
spin5           &  125 &  $\zESCJnoG$&  55.9017 & 55.9017 &    \_   &     \_ 
&    0.17   &  309.35   &    \_   &   \_   \\
&  375 & $\zCESCJnoG$&  55.9017 & 55.9017 &         &        &    0.10   &  
309.61   &         &    \\ \hline
G\_150\_025     &  150 &  $\zESCJnoG$&  23.7185 & 23.4720 &    \_   &     \_ 
&    3.34   & 2049.99   &    \_   &    \_   \\
& 2835 & $\zCESCJnoG$&  23.7185 & 23.4720 &         &        &    2.81   & 
1461.60   &         &    \\ \hline
keller4         &  171 &  $\zESCJnoG$&  14.0122 & 13.4659 &    \_   &     \_ 
&   10.05   & 5386.19   &    \_   &    \_   \\
& 5100 & $\zCESCJnoG$&  14.0122 & 13.4659 &         &        &    8.85   & 
4367.68   &         &    \\ \hline
\end{tabular}
\end{small}
\end{center}
\end{table}

Computing the $k$-th level of the ESH and the CESH by including all ESCs of 
order $k$ is beyond reach rather soon, so in the next experiments we want to 
include the ESCs only for some subgraphs of a given order $k$. In order to 
determine the set $J$ of subgraphs for which to include the ESCs we follow the 
approach of 
Gaar and Rendl~\cite{GaarRendl,GaarRendlFull}. In particular, we start with 
$J=\emptyset$ and iteratively solve an SDP for computing the Lov\'{a}sz theta 
	function (either~\eqref{lovaszTheta_nPlus1} and~\eqref{lovaszTheta_n}) with 
	the already determined ESCs induced by $J$. Then we use the optimal 
	solution of the SDP in order to search for violated ESCs. To find 
	potentially violated subgraphs we perform a heuristic search among all 
	subgraphs that tries to minimize the inner product of the optimal solution 
	corresponding a subgraph and certain matrices (e.g., matrices that induce 
	facets of $\STAB^{2}(\GZeroi{k}))$. We refer  
	to~\cite{GaarRendl,GaarRendlFull} for more details.
We 
perform 10 iterations with including at most 200 ESCs of order~$k$ in each 
iteration, so in the end for each graph and for each $k$ we have a set~$J$ of 
at most 2000 ESCs. Of course it makes a difference whether we do the search 
starting from~\eqref{lovaszTheta_nPlus1} and~\eqref{lovaszTheta_n} as different 
subgraphs might be violated. 
We denote by $J_\mathcal{E}$ and $J_\mathcal{C}$ the set of subsets 
obtained by using~\eqref{lovaszTheta_nPlus1} and~\eqref{lovaszTheta_n} in order 
to search for violated subgraphs.
The used sets $J_\mathcal{E}$ and $J_\mathcal{C}$ are available 
online at \url{https://arxiv.org/src/2003.13605/anc}.
Table~\ref{tab:ttfSearchESCNrESC} summarizes the cardinalities of 
$J_\mathcal{E}$ 
and $J_\mathcal{C}$. 
The values of $\zESCJ$ and $\zCESCJ$ and the running time 
for the sets $J=J_\mathcal{E}$ can be found in Table~\ref{tab:ttfSearchESC} and 
Table~\ref{tab:ttfSearchESCTime}. 
The analogous computational results when considering $J=J_\mathcal{C}$ are 
presented in 
 Table~\ref{tab:ttfSearchCESC} and Table~\ref{tab:ttfSearchCESCTime}.

First, observe in Table~\ref{tab:ttfSearchESCNrESC} that the cardinality 
of 
$J_\mathcal{C}$ 
is typically larger 
compared to the cardinality of $J_\mathcal{E}$. 
This is plausible, because due to the 
additional row and column in~\eqref{lovaszTheta_nPlus1} and the SDP constraint 
in this formulation some ECSs might be satisfied, which are violated in the 
version with~\eqref{lovaszTheta_n}. 

When we turn to the values of $\zESCJ$ and $\zCESCJ$ in 
Table~\ref{tab:ttfSearchESC} and Table~\ref{tab:ttfSearchCESC} we observe 
for both $J=J_\mathcal{E}$ and $J=J_\mathcal{C}$ 
that (a) the 
larger $k$ becomes, the better the bounds are, 
(b) for $k=0$, so for computing~$\vartheta(G)$, we have 
$\zESCJ = \zCESCJ$ as expected, 
(c) for a fixed set $J$ we have $\zESCJ \leqslant \zCESCJ$ in accordance with 
the theory derived earlier and 
(d) typically the difference between  $\zESCJ$ and $\zCESCJ$ increases with 
increasing $k$.
This behavior is observable for both $J=J_\mathcal{E}$ and $J=J_\mathcal{C}$,  
hence the choice of the set $J$ has no significant influence on the 
behavior of the values of $\zESCJ$ and $\zCESCJ$.

However, we observe that usually the values of $\zESCJ$ for $J_\mathcal{E}$ are 
the best bounds, 
then $\zESCJ$ for $J_\mathcal{C}$ are the second best bounds,
$\zCESCJ$ for $J_\mathcal{C}$ are the third best bounds and 
$\zCESCJ$ for $J_\mathcal{E}$ yields the worst bounds - even if the differences 
are typically very small. This behavior is not surprising, because we know that 
for a fixed set $J$ we have $\zESCJ \leqslant \zCESCJ$ and it makes sense that 
the final bounds obtained are better when 
using the same formulation of $\vartheta(G)$ to obtain the bounds that was used 
to obtain $J$. 

Looking at the running times in Table~\ref{tab:ttfSearchESCTime} and 
Table~\ref{tab:ttfSearchCESCTime} we see that our expectations are not met: 
Even 
though the order of the matrix variable and the number of constraints of 
the 
SDP to compute $\zCESCJ$ are smaller than those to compute $\zESCJ$, the 
running 
times are typically larger. So apparently the highly sophisticated interior 
point solver 
MOSEK can deal better with $\zESCJ$. If we compare the running times for the 
set $J = J_\mathcal{E}$ and $J = J_\mathcal{C}$ we see that the running 
times for 
$J_\mathcal{E}$ typically are shorter,  but there are also instances 
(e.g., G\_100\_025 for $k=6$) where the computation of both $\zESCJnoG$ and 
$\zCESCJnoG$ is faster for $J=J_\mathcal{C}$ than it is for $J=J_\mathcal{E}$.

As a result, we confirm that tightening the Lov\'asz theta function 
towards the stability number with the help of ESCs
typically works better when
starting from the Lov\'asz theta function 
formulation~\eqref{lovaszTheta_nPlus1} (as it is done in the ESH) as it 
does when starting with the formulation~\eqref{lovaszTheta_n} (as it is done in 
the CESH), even 
though this is not obvious at 
first sight as the latter SDPs are smaller.
However, in some cases it can be advantageous to use the CESC, but 
then also 
the subset $J$ should be determined using~\eqref{lovaszTheta_n}.

\section{Conclusions}
\label{sec:conclusions}

In this paper we derived two new SDP hierarchies
from the Lov\'asz theta function
towards the stability 
number.
The classical ESH from the  
literature starts from the SDP~\eqref{lovaszTheta_nPlus1} and adds ESCs. We 
introduced the new CESH starting from~\eqref{lovaszTheta_n} and including ESCs. 
We proved that 
this new hierarchy has some same properties as the 
ESH.
Moreover, we showed that the bounds based on the ESH are at least as good 
as those from the CESH - not only for including all ESCs of a certain order, 
but also for including only some of them.

We also newly introduced SESCs, which are 
a more 
natural formulation of exactness for~\eqref{lovaszTheta_n}. 
Including them 
into~\eqref{lovaszTheta_n} 
yields the new SESH. Even though SESCs are more 
intuitive, the bounds based on the CESH and the SESH coincide.

In our computational results with an off-the-shelve interior point solver we 
typically obtain the best bounds with the fastest running times when using 
the ESH. However, for some instances using the CESH is beneficial.

It would be interesting to derive a specialized solver for the CESH as it was 
done by Gaar and Rendl~\cite{GaarRendl,GaarRendlFull} for the ESH. 
They dualize 
the ESCs, use the bundle method and instead of solving a huge SDP with all 
ESCs, they iterate 
and solve~\eqref{lovaszTheta_nPlus1} with a modified 
objective function in each iteration. Since~\eqref{lovaszTheta_n} has a smaller 
matrix order and fewer constraints, this approach  
presumably works even better for the CESH.
Such a solver allows to compare the running times for the ESH and the CESH 
in a more sophisticated way. 

Another open question is the more precise relationship of the ESH and 
the CESH. In this paper we have shown that $\zCESCHG{k} \geqslant \zESCHG{k}$ 
holds for all $k \in \{1, \dots, n\}$. It would be interesting to know if there 
is some constant $\ell \geqslant 1$ such that $\zESCHG{k} \geqslant 
\zCESCHG{k+\ell}$ holds for all graphs $G$ and for all $k \in \{1, \dots, n\}$, 
so 
such that it suffices to add $\ell$ levels to the CESH to reach the quality of 
the ESH.

Finally, it would be interesting to investigate which implications it has for 
the ESH and the CESH to induce the positive semidefiniteness constraint not for 
the whole matrix $X$, but only for a submatrix of $X$ like it has been done in 
the recent work~\cite{blekherman2022sparse}.

\bibliographystyle{amsplain}
\bibliography{papers}

\begin{table}[hb]
\caption{The number of included ESCs $|J_\mathcal{E}|$ and $|J_\mathcal{C}|$ 
for $J=J_\mathcal{E}$ and $J=J_\mathcal{C}$ for the computations 
of 
Table~\ref{tab:ttfSearchESC}
and 
Table~\ref{tab:ttfSearchCESC}
, respectively
}    
\label{tab:ttfSearchESCNrESC}   
\begin{center}
\begin{small}
\setlength{\tabcolsep}{4pt}
\begin{tabular}{|l|r|r|rrrrrr|}
\hline
Name& $n$ & & \multicolumn{6}{c|}{$|J|$ for subgraphs of order $k$}\\
& $m$ &  & $k=0$ & $k=2$ & $k=3$ 
& $k=4$ 
& $k=5$ & $k=6$ \\    
\hline 
HoG\_34272      &     9 &$|J_\mathcal{E}|$& 0 &    9 &   57 &  116 &  126 &   
84\\
&    17 &$|J_\mathcal{C}|$& 0 &    8 &   56 &  110 &  120 &   84 \\ \hline
HoG\_15599      &    20 &$|J_\mathcal{E}|$& 0 &    0 &  141 & 1428 & 2000 & 
2000\\
&    44 &$|J_\mathcal{C}|$& 0 &    0 &  138 & 1240 & 1977 & 2000 \\ \hline
CubicVT26\_5    &    26 &$|J_\mathcal{E}|$& 0 &    0 &  515 & 1189 & 1824 & 
2000\\
&    39 &$|J_\mathcal{C}|$& 0 &    0 &  458 & 1761 & 2000 & 1515 \\ \hline
HoG\_34274      &    36 &$|J_\mathcal{E}|$& 0 &   25 &  823 & 1700 & 2000 & 
2000\\
&    72 &$|J_\mathcal{C}|$& 0 &   24 &  704 & 1593 & 1930 & 2000 \\ \hline
HoG\_6575       &    45 &$|J_\mathcal{E}|$& 0 &    0 &  260 & 1025 & 1378 & 
1785\\
&   225 &$|J_\mathcal{C}|$& 0 &    0 &  268 & 1439 & 1563 & 1490 \\ \hline
MANN\_a9        &    45 &$|J_\mathcal{E}|$& 0 &    0 &  718 & 1102 & 1449 & 
2000\\
&    72 &$|J_\mathcal{C}|$& 0 &    0 &  734 & 1750 & 1950 & 2000 \\ \hline
Circulant47\_30 &    47 &$|J_\mathcal{E}|$& 0 &    0 &  827 & 1337 & 1635 & 
2000\\
&   282 &$|J_\mathcal{C}|$& 0 &    0 &  761 & 1276 & 1796 & 2000 \\ \hline
G\_50\_025      &    50 &$|J_\mathcal{E}|$& 0 &   82 &  413 &  707 & 1146 & 
2000\\
&   308 &$|J_\mathcal{C}|$& 0 &   88 &  521 &  928 & 1645 & 2000 \\ \hline
G\_60\_025      &    60 &$|J_\mathcal{E}|$& 0 &   93 &  492 &  901 & 1366 & 
2000\\
&   450 &$|J_\mathcal{C}|$& 0 &   96 &  486 & 1233 & 1665 & 2000 \\ \hline
Paley61         &    61 &$|J_\mathcal{E}|$& 0 &    0 &    0 &    0 &    0 &   
48\\
&   915 &$|J_\mathcal{C}|$& 0 &    0 &    0 &    0 &    0 &   37 \\ \hline
hamming6\_4     &    64 &$|J_\mathcal{E}|$& 0 &  247 & 1665 & 2000 & 2000 & 
2000\\
&  1312 &$|J_\mathcal{C}|$& 0 &  251 & 1579 & 1970 & 1955 & 2000 \\  \hline
HoG\_34276      &    72 &$|J_\mathcal{E}|$& 0 &   49 & 1402 & 1415 & 1873 & 
2000\\
&   144 &$|J_\mathcal{C}|$& 0 &   76 &  602 & 1398 & 1916 & 2000 \\ \hline
G\_80\_050      &    80 &$|J_\mathcal{E}|$& 0 &  158 &  704 & 1132 & 1854 & 
2000\\
&  1620 &$|J_\mathcal{C}|$& 0 &  220 & 1391 & 1766 & 2000 & 2000 \\ \hline
G\_100\_025     &   100 &$|J_\mathcal{E}|$& 0 &  228 &  590 &  901 & 1658 & 
2000\\
&  1243 &$|J_\mathcal{C}|$& 0 &  235 & 1197 & 1630 & 1961 & 2000 \\ \hline
spin5           &   125 &$|J_\mathcal{E}|$& 0 &    0 & 1204 & 1975 & 2000 & 
2000\\
&   375 &$|J_\mathcal{C}|$& 0 &    0 &  982 & 1829 & 2000 & 2000 \\ \hline
G\_150\_025     &   150 &$|J_\mathcal{E}|$& 0 &  275 &  496 &  804 & 1759 & 
2000\\
&  2835 &$|J_\mathcal{C}|$& 0 &  338 &  718 & 1474 & 1969 & 2000 \\ \hline
keller4         &   171 &$|J_\mathcal{E}|$& 0 &  482 & 1332 & 1959 & 2000 & 
2000\\
&  5100 &$|J_\mathcal{C}|$& 0 &  457 & 1630 & 1931 & 2000 & 2000 \\ \hline
G\_200\_025     &   200 &$|J_\mathcal{E}|$& 0 &  307 &  688 &  884 & 1498 & 
2000\\
&  4905 &$|J_\mathcal{C}|$& 0 &  345 &  812 & 1398 & 2000 & 2000 \\ \hline
brock200\_1     &   200 &$|J_\mathcal{E}|$& 0 &  325 &  571 &  849 & 1406 & 
2000\\
&  5066 &$|J_\mathcal{C}|$& 0 &  365 &  673 & 1395 & 1958 & 2000 \\ \hline
c\_fat200\_5    &   200 &$|J_\mathcal{E}|$& 0 & 1860 & 1913 & 2000 & 2000 & 
2000\\
& 11427 &$|J_\mathcal{C}|$& 0 & 1827 & 1999 & 2000 & 2000 & 2000 \\ \hline
sanr200\_0\_9   &   200 &$|J_\mathcal{E}|$& 0 &  267 &  530 &  844 & 1483 & 
2000\\
&  2037 &$|J_\mathcal{C}|$& 0 &  337 &  636 & 1252 & 2000 & 2000 \\ \hline
\end{tabular}
\end{small}
\end{center}
\end{table}

\begin{table}[ht]
\caption{The values of $\zESCJ$ and $\zCESCJ$ for 
different graphs $G$ and sets $J = J_\mathcal{E}$ 
for subgraphs of order $k$ for $k \in \{0, 2, 3, 4, 5, 6\}$}
\label{tab:ttfSearchESC}    
\begin{center}
    \begin{small}
    \setlength{\tabcolsep}{3pt}
\begin{tabular}{|l|r|r|rrrrrr|}
    \hline
    Name & $n$ & & \multicolumn{6}{c|}{$J = J_\mathcal{E}$ for subgraphs of 
    order 
    $k$}\\
     &  $m$ &  & $k=0$ & $k=2$ & $k=3$ 
    & $k=4$ 
    & $k=5$ & $k=6$ \\    
    \hline 
    HoG\_34272  &   9  & 
            $\zESCJnoG$&     3.3380   
    &     3.2729   
    &     
    3.0605   &     3.0000   &     3.0000   &     3.0000   \\
    &  17           &    $\zCESCJnoG$&     3.3380   &     3.2763   &     
    3.1864   
    &     3.0000   &     3.0000   &     3.0000   \\ \hline 
                       HoG\_15599  &   20  & 
             $\zESCJnoG$&     7.8202   &     7.8202   
    &     
    7.4437   &     7.0000   &     7.0000   &     7.0000   \\
    &  44         &    $\zCESCJnoG$&     7.8202   &     7.8202   &     
    7.4771   
    &     7.4667   &     7.3458   &     7.0000   \\ \hline 
                     CubicVT26\_5  &   26  & 
             $\zESCJnoG$&    11.8171   &    11.8171   
    &    
    10.9961   &    10.7210   &    10.4214   &    10.3357   \\
    &   39          &    $\zCESCJnoG$&    11.8171   &    11.8171   &    
    11.0037   
    &    11.0035   &    10.7778   &    10.6519   \\ \hline 
                      HoG\_34274  &   36  & 
          $\zESCJnoG$&    13.2317   &    13.0915   
    &    
    12.3174   &    12.0525   &    12.0000   &    12.0000   \\
    &   72        &    $\zCESCJnoG$&    13.2317   &    13.1052   &    
    12.7491   
    &    12.2346   &    12.0000   &    12.0000   \\ \hline 
                        HoG\_6575  &   45  & 
           $\zESCJnoG$&    15.0530   &    15.0530   
    &    
    14.3178   &    14.0257   &    13.8179   &    12.7257   \\
    &  225        &    $\zCESCJnoG$&    15.0530   &    15.0530   &    
    14.3178   
    &    14.1817   &    14.1489   &    13.4104   \\ \hline 
                    MANN\_a9  &   45  & 
           $\zESCJnoG$&    17.4750   &    17.4750   
    &    
    17.1203   &    17.0727   &    16.9964   &    16.8635   \\
    &    72         &    $\zCESCJnoG$&    17.4750   &    17.4750   &    
    17.1644   
    &    17.1163   &    17.0654   &    17.0342   \\ \hline 
                 Circulant47\_30  &   47  & 
            $\zESCJnoG$&    14.3022   &    14.3022   
    &    
    13.6103   &    13.1817   &    13.1806   &    13.0734   \\
    &  282     &    $\zCESCJnoG$&    14.3022   &    14.3022   &    
    13.6172   
    &    13.2008   &    13.1943   &    13.0907   \\ \hline 
                      G\_50\_025  &   50  & 
           $\zESCJnoG$&    13.5642   &    13.4554   
    &    
    13.1310   &    12.9420   &    12.7749   &    12.6210   \\
    &  308        &    $\zCESCJnoG$&    13.5642   &    13.4555   &    
    13.2743   
    &    13.1118   &    12.9279   &    12.7287   \\ \hline 
     G\_60\_025  &   60  & 
         $\zESCJnoG$&    14.2815   &    14.2013   
    &    
    14.0450   &    13.8738   &    13.6876   &    13.6702   \\
    &   450        &    $\zCESCJnoG$&    14.2815   &    14.2013   &    
    14.1038   
    &    13.9800   &    13.7834   &    13.7648   \\ \hline 
                     Paley61  &   61  & 
          $\zESCJnoG$&     7.8102   &     7.8102   
    &     
    7.8102   &     7.8102   &     7.8102   &     7.7480   \\
    &     915        &    $\zCESCJnoG$&     7.8102   &     7.8102   &     
    7.8102   
    &     7.8102   &     7.8102   &     7.7720   \\ \hline 
                      hamming6\_4  &   64  & 
          $\zESCJnoG$&     5.3333   &     4.0000   
    &     
    4.0000   &     4.0000   &     4.0000   &     4.0000   \\
    &    1312          &    $\zCESCJnoG$&     5.3333   &     4.0000   &     
    4.0000   
    &     4.0000   &     4.0000   &     4.0000   \\ \hline 
                      HoG\_34276  &   72  & 
           $\zESCJnoG$&    26.4635   &    26.1831   
    &    
    25.5429   &    24.8186   &    24.1331   &    24.1348   \\
    &    144        &    $\zCESCJnoG$&    26.4635   &    26.2105   &    
    25.9856   
    &    25.6362   &    24.8086   &    24.7730   \\ \hline 
                  G\_80\_050  &   80  & 
         $\zESCJnoG$&     9.4353   &     9.3812   
    &     
    9.3775   &     9.3521   &     9.3152   &     9.2633   \\
    &    1620       &    $\zCESCJnoG$&     9.4353   &     9.3812   &     
    9.3790   
    &     9.3626   &     9.3364   &     9.2949   \\ \hline 
                     G\_100\_025  &  100  & 
          $\zESCJnoG$&    19.4408   &    19.2866   
    &    
    19.2606   &    19.2302   &    19.1807   &    19.1196   \\
    &   1243     &    $\zCESCJnoG$&    19.4408   &    19.2866   &    
    19.2692   
    &    19.2516   &    19.2153   &    19.1630   \\ \hline 
                            spin5  &  125  & 
         $\zESCJnoG$&    55.9017   &    55.9017   
    &    
    50.4661   &    50.1027   &    50.0000   &    50.0000   \\
    &   375        &    $\zCESCJnoG$&    55.9017   &    55.9017   &    
    51.8181   
    &    50.6352   &    50.0000   &    50.0081   \\ \hline 
                     G\_150\_025  &  150  & 
          $\zESCJnoG$&    23.7185   &    23.5355   
    &    
    23.4744   &    23.4693   &    23.4637   &    23.4555   \\
    &   2835      &    $\zCESCJnoG$&    23.7185   &    23.5355   &    
    23.4753   
    &    23.4704   &    23.4663   &    23.4602   \\ \hline 
                          keller4  &  171  & 
         $\zESCJnoG$&    14.0122   &    13.7260   
    &    
    13.5252   &    13.4909   &    13.4786   &    13.4801   \\
    &    5100       &    $\zCESCJnoG$&    14.0122   &    13.7261   &    
    13.5253   
    &    13.4909   &    13.4786   &    13.4811   \\ \hline 
                     G\_200\_025  &  200  & 
         $\zESCJnoG$&    28.2165   &    28.0436   
    &    
    27.9630   &    27.9427   &    27.9326   &    27.9345   \\
    &    4905        &    $\zCESCJnoG$&    28.2165   &    28.0436   &    
    27.9630   
    &    27.9427   &    27.9333   &    27.9354   \\ \hline 
                      brock200\_1  &  200  & 
          $\zESCJnoG$&    27.4566   &    27.2969   
    &    
    27.2250   &    27.2036   &    27.1949   &    27.1925   \\
    &   5066        &    $\zCESCJnoG$&    27.4566   &    27.2969   &    
    27.2250   
    &    27.2040   &    27.1955   &    27.1937   \\ \hline 
                     c\_fat200\_5  &  200  & 
        $\zESCJnoG$&    60.3453   &    60.3453   
    &    
    58.0000   &    58.0000   &    58.0000   &    58.0000   \\
    &  11427       &    $\zCESCJnoG$&    60.3453   &    60.3453   &    
    58.0142   
    &    58.0000   &    58.0000   &    58.0000   \\ \hline 
                    sanr200\_0\_9  &  200  & 
         $\zESCJnoG$&    49.2735   &    49.0388   
    &    
    48.9195   &    48.8546   &    48.7465   &    48.7206   \\
    &    2037      &    $\zCESCJnoG$&    49.2735   &    49.0388   &    
    48.9312   
    &    48.8811   &    48.8137   &    48.8035   \\ \hline 
\end{tabular}
\end{small}
\end{center}
\end{table}

\begin{table}[ht]
    \caption{The running times for the results of Table~\ref{tab:ttfSearchESC}}
    \label{tab:ttfSearchESCTime}   
\begin{center}
    \begin{small}
    \setlength{\tabcolsep}{4pt}
\begin{tabular}{|l|r|r|rrrrrr|}
    \hline
    Name & $n$ & & \multicolumn{6}{c|}{$J = J_\mathcal{E}$ for subgraphs of 
    order 
    $k$}\\    
     &  $m$ &  & $k=0$ & $k=2$ & $k=3$ 
    & $k=4$ 
    & $k=5$ & $k=6$ \\
    \hline 
                    HoG\_34272  &    9  & 
            $\zESCJnoG$&    0.26   &    0.07   &    
    0.08   
    &    0.18   &    0.36   &    0.38   \\
    &  17         &    $\zCESCJnoG$&    0.04   &    0.04   &    0.08   &    
    0.15   &    0.29   &    0.36   \\ \hline 
                       HoG\_15599  &   20  & 
           $\zESCJnoG$&    0.04   &    0.04   &    
    0.21   
    &   95.64   &  869.93   & 1624.17   \\
    &    44         &    $\zCESCJnoG$&    0.04   &    0.03   &    0.20   &  
    123.77   & 1005.30   & 2413.69   \\ \hline 
                     CubicVT26\_5  &   26  & 
          $\zESCJnoG$&    0.04   &    0.03   &    
    2.18   
    &   73.75   &  727.67   & 2307.93   \\
    &    39         &    $\zCESCJnoG$&    0.04   &    0.03   &    1.88   &   
    61.06   &  822.95   & 2980.21   \\ \hline 
                      HoG\_34274  &   36  & 
            $\zESCJnoG$&    0.04   &    0.06   &   
    12.79   
    &  349.74   & 1035.69   & 2205.90   \\
    &    72       &    $\zCESCJnoG$&    0.04   &    0.06   &   11.20   &  
    284.22   &  977.60   & 2773.10   \\ \hline 
                        HoG\_6575  &   45  & 
           $\zESCJnoG$&    0.05   &    0.05   &    
    0.78   
    &   43.73   &  243.83   & 1163.78   \\
    &   225        &    $\zCESCJnoG$&    0.05   &    0.04   &    0.77   &   
    45.97   &  275.98   & 1139.69   \\ \hline 
                    MANN\_a9  &   45  & 
           $\zESCJnoG$&    0.04   &    0.05   &    
    7.09   
    &   76.40   &  465.01   & 2891.53   \\
    &    72        &    $\zCESCJnoG$&    0.05   &    0.04   &    6.39   &   
    76.05   &  494.98   & 3391.76   \\ \hline 
                 Circulant47\_30  &   47  & 
          $\zESCJnoG$&    0.07   &    0.06   &   
    10.56   
    &  116.63   &  490.99   & 2181.69   \\
    &   282        &    $\zCESCJnoG$&    0.07   &    0.06   &   11.18   &  
    109.49   &  430.87   & 2321.85   \\ \hline 
                      G\_50\_025  &   50  & 
          $\zESCJnoG$&    0.08   &    0.18   &    
    2.43   
    &   23.40   &  200.74   & 2377.75   \\
    &   308         &    $\zCESCJnoG$&    0.08   &    0.18   &    2.49   &   
    22.07   &  209.83   & 2606.79   \\ \hline 
     G\_60\_025 &   60  & 
          $\zESCJnoG$&    0.13   &    0.29   &    
    4.36   
    &   43.88   &  326.96   & 2495.65   \\
    &   450        &    $\zCESCJnoG$&    0.12   &    0.27   &    4.19   &   
    52.47   &  380.22   & 2323.79   \\ \hline 
                     Paley61  &   61  & 
           $\zESCJnoG$&    0.23   &    0.22   &    
    0.22   
    &    0.23   &    0.22   &    0.67   \\
    &     915       &    $\zCESCJnoG$&    0.17   &    0.16   &    0.17   &    
    0.18   &    0.16   &    0.56   \\ \hline 
                      hamming6\_4  &   64  & 
         $\zESCJnoG$&    0.36   &    1.36   &   
    45.30   
    &  133.11   &  318.78   &  802.02   \\
    &    1312        &    $\zCESCJnoG$&    0.29   &    1.05   &   29.49   &  
    113.16   &  366.63   &  804.42   \\ \hline 
                      HoG\_34276  &   72  & 
           $\zESCJnoG$&    0.07   &    0.12   &   
    63.26   
    &  218.38   & 1845.60   & 5152.23   \\
    &    144     &    $\zCESCJnoG$&    0.07   &    0.12   &   60.17   &  
    227.44   & 1978.86   & 6327.47   \\ \hline 
                  G\_80\_050  &   80  & 
           $\zESCJnoG$&    0.89   &    1.90   &   
    12.99   
    &   87.98   &  713.59   & 1391.32   \\
    &     1620     &    $\zCESCJnoG$&    0.92   &    1.94   &   14.02   &   
    96.16   &  624.19   & 1497.57   \\ \hline 
                     G\_100\_025  &  100  & 
         $\zESCJnoG$&    0.64   &    1.77   &   
    10.03   
    &   69.21   &  691.72   & 2768.81   \\
    &      1243      &    $\zCESCJnoG$&    0.56   &    1.69   &   10.67   &   
    68.02   &  690.08   & 2709.48   \\ \hline 
                            spin5  &  125  & 
          $\zESCJnoG$&    0.16   &    0.17   &   
    37.22   
    &  342.58   &  947.04   & 1972.30   \\
    &   375         &    $\zCESCJnoG$&    0.10   &    0.10   &   31.63   &  
    296.17   & 1110.95   & 3237.65   \\ \hline 
                     G\_150\_025  &  150  & 
        $\zESCJnoG$&    3.27   &    6.70   &   
    19.19   
    &   91.37   &  957.42   & 2867.07   \\
    &  2835      &    $\zCESCJnoG$&    3.05   &    6.50   &   16.98   &   
    77.65   & 1041.79   & 3101.88   \\ \hline 
                          keller4  &  171  & 
        $\zESCJnoG$&   11.36   &   24.40   &  
    178.18   
    &  749.89   & 1644.06   & 4040.91   \\
    &     5100        &    $\zCESCJnoG$&    9.51   &   24.76   &  150.29   &  
    737.29   & 2023.90   & 4937.69   \\ \hline 
                     G\_200\_025  &  200  & 
        $\zESCJnoG$&   10.04   &   18.35   &   
    54.11   
    &  141.27   &  853.42   & 3252.01   \\
    &     4905         &    $\zCESCJnoG$&   10.05   &   18.56   &   57.54   &  
    143.34   &  972.45   & 4075.58   \\ \hline 
                      brock200\_1  &  200  & 
         $\zESCJnoG$&   11.81   &   20.90   &   
    44.77   
    &  151.79   &  782.03   & 3350.39   \\
    &    5066       &    $\zCESCJnoG$&   10.78   &   20.04   &   43.04   &  
    144.64   &  855.91   & 3929.63   \\ \hline 
                     c\_fat200\_5  &  200  & 
         $\zESCJnoG$&   49.31   &  177.02   &  
    563.46   
    &  755.78   & 3140.87   & 3714.74   \\
    &    11427       &    $\zCESCJnoG$&   37.30   &  156.36   &  653.76   &  
    767.95   & 3101.51   & 3333.48   \\ \hline 
                    sanr200\_0\_9  &  200  & 
       $\zESCJnoG$&    2.17   &    4.19   &   
    17.45   
    &   86.09   &  867.56   & 4353.92   \\
    &     2037     &    $\zCESCJnoG$&    1.76   &    4.43   &   16.98   &   
    81.76   &  881.66   & 4631.88   \\ \hline 
\end{tabular}
\end{small}
\end{center}
\end{table}

\begin{table}[ht]
\caption{The values of $\zESCJ$ and $\zCESCJ$ for 
    different graphs $G$ and sets $J = J_\mathcal{C}$ 
    for subgraphs of order $k$ for $k \in \{0, 2, 3, 4, 5, 6\}$}
    \label{tab:ttfSearchCESC}    
    \begin{center}
        \begin{small}
        \setlength{\tabcolsep}{3pt}
        \begin{tabular}{|l|r|r|rrrrrr|}
            \hline
               Name & $n$ & & \multicolumn{6}{c|}{$J = J_\mathcal{C}$ for 
               subgraphs 
                of 
                order $k$}\\
             &$m$ &  & $k=0$ & $k=2$ & $k=3$ 
            & $k=4$ 
            & $k=5$ & $k=6$ \\    
            \hline 
                            HoG\_34272  & 
               9  & 
                   $\zESCJnoG$&     3.3380   &     
            3.2729   &     
            3.0605   &     3.0000   &     3.0000   &     3.0000   \\
            &    17         &    $\zCESCJnoG$&     3.3380   &     3.2763   
            &     
            3.1765   
            &     3.0000   &     3.0000   &     3.0000   \\ \hline 
                               HoG\_15599  &   
            20  & 
                    $\zESCJnoG$&     7.8202   &     
            7.8202   &     
            7.4472   &     7.3871   &     7.0000   &     7.0000   \\
            &   44        &    $\zCESCJnoG$&     7.8202   &     7.8202   &     
            7.4761   
            &     7.4291   &     7.2500   &     7.0000   \\ \hline 
                             CubicVT26\_5  &   
            26  & 
                  $\zESCJnoG$&    11.8171   &    
            11.8171   &    
            11.0035   &    10.9932   &    10.5956   &    10.3201   \\
            &     39       &    $\zCESCJnoG$&    11.8171   &    11.8171   &    
            11.0035   
            &    11.0035   &    10.7189   &    10.5727   \\ \hline 
                              HoG\_34274  &   
            36  & 
                   $\zESCJnoG$&    13.2317   &    
            13.0915   &    
            12.3066   &    12.1338   &    12.0000   &    12.0000   \\
            &    72        &    $\zCESCJnoG$&    13.2317   &    13.1052   &    
            12.6217   
            &    12.3998   &    12.0000   &    12.0000   \\ \hline 
                                HoG\_6575  &   
            45  & 
                 $\zESCJnoG$&    15.0530   &    
            15.0530   &    
            14.3178   &    14.1594   &    14.0495   &    12.9931   \\
            &    225       &    $\zCESCJnoG$&    15.0530   &    15.0530   &    
            14.3178   
            &    14.1595   &    14.0791   &    13.2063   \\ \hline 
                            MANN\_a9  &   45  
            & 
                   $\zESCJnoG$&    17.4750   &    
            17.4750   &    
            17.1332   &    17.0762   &    17.0009   &    16.9167   \\
            &     72      &    $\zCESCJnoG$&    17.4750   &    17.4750   &    
            17.1471   
            &    17.1092   &    17.0591   &    17.0349   \\ \hline 
                         Circulant47\_30  &   
            47  & 
                 $\zESCJnoG$&    14.3022   &    
            14.3022   &    
            13.6188   &    13.1845   &    13.1827   &    13.0516   \\
            &    282        &    $\zCESCJnoG$&    14.3022   &    14.3022   &    
            13.6233   
            &    13.1934   &    13.1887   &    13.0598   \\ \hline 
                              G\_50\_025  &   
            50  & 
                 $\zESCJnoG$&    13.5642   &    
            13.4554   &    
            13.1225   &    12.9735   &    12.7355   &    12.6194   \\
            &   308      &    $\zCESCJnoG$&    13.5642   &    13.4555   &    
            13.2253   
            &    13.0803   &    12.8240   &    12.7127   \\ \hline 
             G\_60\_025  &   60  & 
                  $\zESCJnoG$&    14.2815   &    
            14.2013   &    
            14.0466   &    13.9115   &    13.7271   &    13.6885   \\
            &     450       &    $\zCESCJnoG$&    14.2815   &    14.2013   &    
            14.1014   
            &    13.9727   &    13.7845   &    13.7478   \\ \hline 
                             Paley61  &   
            61  & 
                  $\zESCJnoG$&     7.8102   &     
            7.8102   &     
            7.8102   &     7.8102   &     7.8102   &     7.7803   \\
            &   915      &    $\zCESCJnoG$&     7.8102   &     7.8102   &     
            7.8102   
            &     7.8102   &     7.8102   &     7.7945   \\ \hline 
                              hamming6\_4  &   
            64  & 
               $\zESCJnoG$&     5.3333   &     
            4.0000   &     
            4.0000   &     4.0000   &     4.0000   &     4.0000   \\
            &    1312       &    $\zCESCJnoG$&     5.3333   &     4.0000   
            &     
            4.0000   
            &     4.0000   &     4.0000   &     4.0000   \\ \hline 
                              HoG\_34276  &   
            72  & 
                 $\zESCJnoG$&    26.4635   &    
            26.1831   &    
            25.7450   &    24.9159   &    24.1589   &    24.0977   \\
            &   144       &    $\zCESCJnoG$&    26.4635   &    26.2105   &    
            26.0414   
            &    25.5734   &    24.7075   &    24.4902   \\ \hline 
                          G\_80\_050  &   80  
            & 
                $\zESCJnoG$&     9.4353   &     
            9.3814   &     
            9.3741   &     9.3566   &     9.3182   &     9.2835   \\
            &     1620       &    $\zCESCJnoG$&     9.4353   &     9.3814   
            &     
            9.3777   
            &     9.3658   &     9.3367   &     9.3114   \\ \hline 
                             G\_100\_025  &  
            100  & 
               $\zESCJnoG$&    19.4408   &    
            19.2892   &    
            19.2639   &    19.2255   &    19.1638   &    19.1260   \\
            &    1243         &    $\zCESCJnoG$&    19.4408   &    19.2892   
            &    
            19.2716   
            &    19.2498   &    19.2015   &    19.1730   \\ \hline 
                                    spin5  &  
            125  & 
                $\zESCJnoG$&    55.9017   &    
            55.9017   &    
            50.6697   &    50.2870   &    50.0000   &    50.0000   \\
            &    375      &    $\zCESCJnoG$&    55.9017   &    55.9017   &    
            51.2339   
            &    50.3559   &    50.0000   &    50.0000   \\ \hline 
                             G\_150\_025  &  
            150  & 
               $\zESCJnoG$&    23.7185   &    
            23.5122   &    
            23.4752   &    23.4676   &    23.4641   &    23.4599   \\
            &    2835      &    $\zCESCJnoG$&    23.7185   &    23.5122   &    
            23.4754   
            &    23.4691   &    23.4667   &    23.4634   \\ \hline 
                                  keller4  &  
            171  & 
               $\zESCJnoG$&    14.0122   &    
            13.6845   &    
            13.5526   &    13.4896   &    13.4792   &    13.4823   \\
            &      5100      &    $\zCESCJnoG$&    14.0122   &    13.6846   
            &    
            13.5526   
            &    13.4896   &    13.4792   &    13.4826   \\ \hline 
                             G\_200\_025  &  
            200  & 
                $\zESCJnoG$&    28.2165   &    
            28.0139   &    
            27.9675   &    27.9447   &    27.9342   &    27.9313   \\
            &   4905      &    $\zCESCJnoG$&    28.2165   &    28.0139   &    
            27.9675   
            &    27.9449   &    27.9345   &    27.9326   \\ \hline 
                              brock200\_1  &  
            200  & 
               $\zESCJnoG$&    27.4566   &    
            27.2911   &    
            27.2212   &    27.2007   &    27.1950   &    27.1928   \\
            &   5066        &    $\zCESCJnoG$&    27.4566   &    27.2911   &    
            27.2212   
            &    27.2011   &    27.1960   &    27.1941   \\ \hline 
                             c\_fat200\_5  &  
            200  & 
                $\zESCJnoG$&    60.3453   &    
            60.3453   &    
            58.0000   &    58.0000   &    58.0000   &    58.0000   \\
            &    11427        &    $\zCESCJnoG$&    60.3453   &    60.3453   
            &    
            58.0000   
            &    58.0000   &    58.0000   &    58.0000   \\ \hline 
                            sanr200\_0\_9  &  
            200  & 
                $\zESCJnoG$&    49.2735   &    
            49.0222   &    
            48.9106   &    48.8432   &    48.7635   &    48.7174   \\
            &     2037       &    $\zCESCJnoG$&    49.2735   &    49.0222   
            &    
            48.9251   
            &    48.8736   &    48.8124   &    48.7893   \\ \hline 
        \end{tabular}
    \end{small}
    \end{center}
\end{table}

\begin{table}[ht]
    \caption{The running times for the results of 
    Table~\ref{tab:ttfSearchCESC}}
    \label{tab:ttfSearchCESCTime}     
    \begin{center}
        \begin{small}
        \setlength{\tabcolsep}{4pt}
        \begin{tabular}{|l|r|r|rrrrrr|}
            \hline
               Name & $n$ & & \multicolumn{6}{c|}{$J = J_\mathcal{C}$ for 
               subgraphs 
                of 
                order $k$}\\
             & $m$ &  & $k=0$ & $k=2$ & $k=3$ 
            & $k=4$ 
            & $k=5$ & $k=6$ \\    
            \hline 
                            HoG\_34272  & 
               9  & 
                    $\zESCJnoG$&    0.27   &    0.04   
            &    
            0.06   
            &    0.19   &    0.28   &    0.24   \\
            &     17       &    $\zCESCJnoG$&    0.04   &    0.03   &    0.07   
            &    
            0.15   &    0.30   &    0.28   \\ \hline 
                               HoG\_15599  &   
            20  & 
                 $\zESCJnoG$&    0.05   &    0.05   
            &    
            0.22   
            &   76.07   &  919.15   & 1507.33   \\
            &   44    &    $\zCESCJnoG$&    0.04   &    0.03   &    0.21   
            &   
            84.53   &  916.81   & 2409.14   \\ \hline 
                             CubicVT26\_5  &   
            26  & 
                  $\zESCJnoG$&    0.05   &    0.04   
            &    
            1.45   
            &  175.55   &  868.47   & 1183.23   \\
            &    39         &    $\zCESCJnoG$&    0.05   &    0.03   &    
            1.57   
            &  
            174.67   & 1132.52   & 1132.77   \\ \hline 
                              HoG\_34274  &   
            36  & 
                  $\zESCJnoG$&    0.05   &    0.06   
            &    
            8.18   
            &  234.84   &  816.22   & 2301.95   \\
            &    72        &    $\zCESCJnoG$&    0.04   &    0.06   &    9.09   
            &  
            262.06   &  802.13   & 1949.73   \\ \hline 
                                HoG\_6575  &   
            45  & 
                  $\zESCJnoG$&    0.05   &    0.05   
            &    
            0.67   
            &  126.05   &  376.49   &  653.98   \\
            &   225       &    $\zCESCJnoG$&    0.05   &    0.05   &    0.66   
            &  
            126.67   &  439.78   &  711.86   \\ \hline 
                            MANN\_a9  &   45  
            & 
                  $\zESCJnoG$&    0.04   &    0.04   
            &    
            8.14   
            &  246.63   &  997.40   & 2840.47   \\
            &   72      &    $\zCESCJnoG$&    0.06   &    0.04   &    7.46   
            &  
            252.58   & 1091.76   & 3095.15   \\ \hline 
                         Circulant47\_30  &   
            47  & 
                  $\zESCJnoG$&    0.08   &    0.07   
            &    
            9.01   
            &   88.89   &  654.73   & 2551.06   \\
            &  282       &    $\zCESCJnoG$&    0.07   &    0.06   &    8.39   
            &   
            90.83   &  654.91   & 2443.97   \\ \hline 
                              G\_50\_025  &   
            50  & 
                 $\zESCJnoG$&    0.08   &    0.19   
            &    
            3.58   
            &   50.21   &  574.86   & 2578.57   \\
            &    308       &    $\zCESCJnoG$&    0.08   &    0.17   &    3.87   
            &   
            44.95   &  541.70   & 2449.45   \\ \hline 
             G\_60\_025  &   60  & 
                 $\zESCJnoG$&    0.14   &    0.30   
            &    
            4.05   
            &   94.83   &  591.48   & 2722.16   \\
            &      450       &    $\zCESCJnoG$&    0.11   &    0.27   &    
            3.97   
            &  
            106.23   &  566.97   & 2564.81   \\ \hline 
                             Paley61  &   
            61  & 
                  $\zESCJnoG$&    0.24   &    0.22   
            &    
            0.22   
            &    0.21   &    0.22   &    0.59   \\
            &     915    &    $\zCESCJnoG$&    0.17   &    0.16   &    0.17   
            &    
            0.19   &    0.17   &    0.55   \\ \hline 
                              hamming6\_4  &   
            64  & 
              $\zESCJnoG$&    0.36   &    1.19   
            &   
            36.38   
            &  131.06   &  354.15   &  783.43   \\
            &    1312       &    $\zCESCJnoG$&    0.30   &    1.06   &   
            41.24   
            &  
            120.08   &  342.97   &  783.59   \\ \hline 
                              HoG\_34276  &   
            72  & 
                $\zESCJnoG$&    0.07   &    0.15   
            &    
            7.25   
            &  204.30   & 1671.17   & 4388.30   \\
            &   144        &    $\zCESCJnoG$&    0.06   &    0.15   &    8.07   
            &  
            245.84   & 1819.67   & 5242.20   \\ \hline 
                          G\_80\_050  &   80  
            & 
                 $\zESCJnoG$&    0.91   &    3.06   
            &   
            60.68   
            &  232.45   &  656.17   & 1531.83   \\
            &   1620   &    $\zCESCJnoG$&    0.88   &    2.67   &   57.14   
            &  
            219.95   &  689.89   & 1572.06   \\ \hline 
                             G\_100\_025  &  
            100  & 
                $\zESCJnoG$&    0.66   &    1.81   
            &   
            40.46   
            &  249.57   & 1151.43   & 2534.13   \\
            &     1243        &    $\zCESCJnoG$&    0.56   &    1.83   &   
            40.63   
            &  
            269.32   & 1125.60   & 2634.70   \\ \hline 
                                    spin5  &  
            125  & 
                $\zESCJnoG$&    0.17   &    0.17   
            &   
            19.30   
            &  261.38   & 1320.40   & 2529.39   \\
            &     375       &    $\zCESCJnoG$&    0.10   &    0.10   &   
            19.53   
            &  
            251.51   & 1578.82   & 3110.74   \\ \hline 
                             G\_150\_025  &  
            150  & 
              $\zESCJnoG$&    3.42   &    8.32   
            &   
            29.13   
            &  224.75   & 1208.97   & 2954.46   \\
            &     2835       &    $\zCESCJnoG$&    3.07   &    6.90   &   
            27.63   
            &  
            257.91   & 1463.41   & 3219.20   \\ \hline 
                                  keller4  &  
            171  & 
                $\zESCJnoG$&   11.65   &   22.97   
            &  
            200.58   
            &  677.57   & 1484.11   & 5050.42   \\
            &   5100     &    $\zCESCJnoG$&    9.47   &   24.13   &  227.81   
            &  
            711.48   & 1815.60   & 5073.47   \\ \hline 
                             G\_200\_025  &  
            200  & 
                $\zESCJnoG$&   10.02   &   18.98   
            &   
            68.21   
            &  375.73   & 1844.91   & 3730.11   \\
            &    4905      &    $\zCESCJnoG$&   10.07   &   19.69   &   74.55   
            &  
            359.11   & 2011.74   & 3947.52   \\ \hline 
                              brock200\_1  &  
            200  & 
               $\zESCJnoG$&   11.79   &   20.48   
            &   
            53.32   
            &  343.51   & 1596.36   & 3673.71   \\
            &    5066        &    $\zCESCJnoG$&    9.97   &   21.54   &   
            54.54   
            &  
            364.17   & 1801.74   & 3993.45   \\ \hline 
                             c\_fat200\_5  &  
            200  & 
               $\zESCJnoG$&   53.80   &  164.89   
            &  
            702.13   
            & 1060.21   & 3509.77   & 3669.49   \\
            &    11427      &    $\zCESCJnoG$&   37.18   &  159.80   &  
            762.12   
            & 
            1038.43   & 3252.03   & 3764.20   \\ \hline 
                            sanr200\_0\_9  &  
            200  & 
               $\zESCJnoG$&    2.18   &    5.20   
            &   
            20.46   
            &  208.74   & 1615.14   & 4255.90   \\
            &     2037      &    $\zCESCJnoG$&    1.78   &    5.57   &   
            21.65   
            &  
            211.25   & 1635.43   & 4749.57   \\ \hline 
        \end{tabular}
                \end{small}
    \end{center}
\end{table}

\end{document}